\documentclass[11pt]{article}
\usepackage[utf8]{inputenc}
\usepackage[T1]{fontenc}    
\usepackage[round,authoryear]{natbib}
\usepackage[a4paper]{geometry}
\usepackage[english]{babel}
\usepackage{amsmath, amsthm, amssymb, mathtools, amsfonts}
\usepackage{xcolor}
\usepackage{algorithm}
\usepackage{algpseudocode}
\usepackage{booktabs} 
\usepackage{multirow}
\usepackage{graphicx} 
\graphicspath{{./}{./img/}}
\usepackage{csquotes}

\usepackage{microtype}
\usepackage{appendix}  
\usepackage{enumitem} 
\usepackage{bbm}     
\usepackage{hyperref}
\usepackage{placeins}

\newtheorem{theorem}{Theorem}[section]
\newtheorem{lemma}[theorem]{Lemma}
\newtheorem{proposition}[theorem]{Proposition}

\newtheorem{corollary}[theorem]{Corollary}

\newtheorem{remark}[theorem]{Remark}

\newcommand{\E}{\mathbb{E}} 
\newcommand{\R}{\mathbb{R}} 
 
\newcommand{\F}{\mathcal{F}} 
\newcommand{\norm}[1]{\left\lVert#1\right\rVert} 

\begin{document}

\title{
Complexity reduction in online stochastic Newton methods 
with potential $\mathcal{O}(Nd)$ total cost
}


\author{Antoine Godichon-Baggioni$^{(1)}$,   Bruno Portier$^{(2)}$ and    Guillaume Sallé$^{(2)}$ \\
$(1)$ Laboratoire de Probabilités, Statistique et Modélisation,\\
Sorbonne Université, 75005 Paris, France, antoine.godichon\_baggioni@upmc.fr \\
$(2)$ Laboratoire de mathématiques de l'INSA, \\
INSA Rouen Normandie, 76800 Saint-Etienne du Rouvray, France
}

\date{}

\maketitle

\begin{abstract}
Optimizing smooth convex functions in stochastic settings, 
where only noisy estimates of gradients and Hessians are available, 
is a fundamental problem in optimization.
While first-order methods possess a low per-iteration cost, 
their convergence is slow for ill-conditioned problems.
Stochastic Newton methods utilize second-order information to correct for local curvature,
but the $\mathcal{O}(d^3)$ per-iteration cost of computing and inverting a full Hessian,
where $d$ is the problem dimension, is prohibitive in high dimensions.
This paper introduces an online mini-batch stochastic Newton algorithm. 
The method employs a random masking strategy that selects a subset of Hessian columns at each iteration, 
substantially reducing the per-step computational cost.
This approach allows the algorithm, in the mini-batch setting,
to achieve a total computational cost for a single pass over $N$ data points of $\mathcal{O}(Nd)$,
which is comparable to first-order methods while retaining the advantages of second-order information.
We establish the almost sure convergence and asymptotic efficiency of the resulting estimator.
This property is obtained without requiring iterate averaging,
which distinguishes this work from prior analyses.
\end{abstract}


\section{Introduction}
     
Estimating the minimizer $\theta^* \in \R^d$ of a smooth convex function $F:\R^d \rightarrow \R$ is a fundamental problem in applied mathematics and machine learning.
This paper studies this problem within a stochastic online framework, wherein the function $F$ is not directly accessible.
Instead, we assume access to stochastic oracles that provide noisy estimates of the gradient $\nabla F(\theta)$ and Hessian $\nabla^2 F(\theta)$.
This setting is common in large-scale learning, where the objective $F$ is defined as an expectation of a known function $f$,
$F(\theta)=\mathbb{E}_{\xi}[f(\xi,\theta)]$, with respect to a random variable $\xi$.
The minimizer $\theta^*$ is then estimated from sequential samples of $\xi$.

For large-scale problems, particularly when datasets exceed memory capacity or arrive sequentially, 
online algorithms are essential.
First-order methods, notably \emph{stochastic gradient descent} (SGD), 
are widely employed due to their simplicity and low per-iteration computational cost, 
typically $\mathcal{O}(d)$.
Their theoretical properties are well-understood, 
with foundational asymptotic results in \cite{pelletier_asymptotic_2000,duflo_algorithmes_1996} 
and non-asymptotic analyses in \cite{bach_non-asymptotic_2011}. 
We refer to \cite{bottou_optimization_2018} for a comprehensive survey.
However, SGD employs a single scalar step-size, which can lead to slow convergence for ill-conditioned problems.
To address this limitation, a broad class of algorithms, termed \emph{conditioned SGD},
incorporates a pre-conditioning matrix $C_{n-1}$ into the update:
\[
\theta_n = \theta_{n-1} - \alpha_n C_{n-1} g_n(\theta_{n-1}),
\]
where $\alpha_n > 0$ is the step-size and $g_n(\theta_{n-1})$ is a stochastic estimate of the gradient 
$\nabla F(\theta_{n-1})$.
The choice of $C_{n-1}$ dictates the trade-off between convergence performance and computational cost.
Well-known adaptive methods, such as AdaGrad \cite{duchi_adaptive_2011}, RMSProp \cite{hinton_neural_2012}, and Adam \cite{kingma_adam_2017},
utilize diagonal matrices for $C_{n-1}$.
This strategy adapts the step-size for each coordinate but fails to capture the off-diagonal curvature information.

The ideal conditioning, inspired by Newton's method, sets $C_{n-1}$ to approximate the inverse Hessian, 
$(\nabla^2 F(\theta_{n-1}))^{-1}$.
Methods that explicitly estimate this matrix are known as \emph{stochastic Newton methods}.
This pre-conditioning corrects for the local curvature of the objective function, 
rendering the algorithm robust to ill-conditioned problems.
Stochastic Newton methods are capable of achieving \emph{asymptotic efficiency},
a fundamental benchmark for asymptotic optimality introduced in \cite{pelletier_asymptotic_2000, duflo_algorithmes_1996}.
This property signifies that the estimator's asymptotic covariance matrix attains the Cramér-Rao lower bound
in relevant statistical settings; see, e.g., \cite{leluc_asymptotic_2023}.
Asymptotic efficiency can be achieved if the pre-conditioning matrix $C_{n}$ converges almost surely to 
the inverse Hessian at the optimum, $H^{-1} \coloneq (\nabla^2 F(\theta^*))^{-1}$, as $n \to \infty$ 
\cite{leluc_asymptotic_2023}. 
This result relaxes the stronger convergence rate conditions required by previous analyses 
\cite{boyer_asymptotic_2023}.

The principal obstacle to their practical application, however, remains the computational cost.
A straightforward implementation, such as that in \cite{leluc_asymptotic_2023}, 
requires computing and inverting an empirical Hessian at each iteration. 
This is an $\mathcal{O}(d^3)$ operation that is prohibitive in high-dimensional settings ($d \gg 1$).

\paragraph{Related works.}

Significant research has focused on reducing the $\mathcal{O}(d^3)$ computational burden of stochastic Newton methods.
Stochastic Quasi-Newton methods, such as (L)-BFGS variants for the online setting 
\cite{schraudolph_stochastic_2007,moritz_linearly-convergent_2016}, 
circumvent direct Hessian calculations by using gradient differences to build an approximation of the inverse Hessian,
 typically at a cost of $\mathcal O(md)$.
Other approaches avoid forming or inverting the full Hessian by using iterative linear solvers
that only require Hessian-vector products, such as stochastic Newton-CG methods \cite{byrd_use_2011}.

While these methods reduce the per-iteration cost, they generally lack guarantees of asymptotic efficiency, 
as they do not ensure that the pre-conditioning matrix converges to the true inverse Hessian.
 
In the particular case where the Hessian estimates are rank-one matrices, as in generalized linear models,
algorithms with a reduced $\mathcal{O}(d^2)$ cost have been developed.

For instance, \cite{bercu_efficient_2020} introduced an $\mathcal{O}(d^2)$ algorithm that leverages the Sherman-Morrison formula for efficient rank-1 updates \cite{dennis_jr_quasi-newton_1977}.
\cite{boyer_asymptotic_2023} extended this approach to a broader class of models, achieving asymptotic efficiency with $\mathcal{O}(d^2)$ complexity per iteration.

Recently, an online algorithm was introduced by \cite{godichon-baggioni_online_2025}, 
which estimates the inverse Hessian recursively with a $\mathcal{O}(d^2)$ cost per iteration, 
applies to general Hessian structures and possesses established convergence guarantees.
However, this algorithm required an additional averaging step to achieve asymptotic efficiency.

Randomized iterative algorithms for matrix inversion have been proposed to reduce computational cost. 
These methods, which include stochastic variants of quasi-Newton updates, 
compute an approximation of a fixed matrix inverse by repeatedly solving a "sketched," 
or projected, version of the linear matrix equation \cite{gower_randomized_2017}. 
Such algorithms necessarily operate in an offline setting, 
where the target matrix is known and available at each iteration.
This framework is distinct from the online estimation setting considered here, 
where the target Hessian is unknown and must be estimated sequentially from a stream of noisy data. 
Consequently, methods designed for offline iteration are not directly applicable to our problem.

An alternative path to achieving asymptotic efficiency, distinct from second-order conditioning,
is the averaging method.
Introduced independently by \cite{polyak_new_1990} and \cite{ruppert_efficient_1988}, 
and further developed by \cite{polyak_acceleration_1992}, 
this method, commonly known as Polyak-Ruppert averaging,
allows first-order SGD variants to attain the optimal asymptotic covariance.
While powerful, averaging does not inherently adapt to the local geometry 
in the same manner as methods employing Hessian information.

Mini-batching is another important consideration for large-scale applications,
as it can stabilize estimates, leverage parallel architectures, and reduce update overhead
\cite{bottou_optimization_2018}.
In the particular case where the Hessian estimates are rank-one matrices, 
\cite{godichon-baggioni_adaptive_2025} proposed an asymptotically efficient mini-batch stochastic Newton algorithm 
with reduced cost by updating the pre-conditioning matrix less frequently.
However, this approach is not directly applicable to general Hessian structures
and incorporates less information from the Hessian esimates,
as it utilizes only a single sample from the mini-batch for the pre-conditioner update rather than
aggregating information from the entire batch.
For the general case, the algorithm of \cite{godichon-baggioni_online_2025} is asymptotically efficient in the purely online setting (batch size $b=1$), 
but its extension to the mini-batch setting ($b>1$) case remains a challenge.
 
\paragraph{Contributions.}
This paper introduces an asymptotically efficient online mini-batch (or streaming) algorithm, 
the \emph{masked Stochastic Newton Algorithm} (mSNA),
which extends and improves upon the method proposed in \cite{godichon-baggioni_online_2025}.
Our work presents two principal advancements.
First, we establish the almost sure convergence of the mSNA estimator and its asymptotic efficiency 
\emph{without} requiring iterate averaging, 
theoretical guarantees not provided in the prior work.
Second, we develop a practical method with reduced computational cost 
for processing mini-batched Hessian information.
Our algorithm operates in a fully online setting (unlike offline iterative solvers)
and applies to general Hessian structures (rather than only rank-one).
Its computational gains are achieved by deriving the inverse Hessian estimator itself 
from a stochastic gradient descent (SGD) on a specific matrix functional.
We then employ a random masking strategy for this underlying SGD, 
which selects a subset of Hessian columns at each iteration; 
this technique can be seen as an adaptation of SGD with coordinate sampling \cite{leluc_sgd_2022}.
This approach allows the algorithm to process mini-batches of size $b$ with a per-iteration
computational cost of $\mathcal{O}(\ell b d + \ell d^2)$, where $\ell \ll d$ is a rank parameter.
On a dataset of size $N$, the algorithm performs $N/b$ iterations.
By setting the batch size $b=d$ and $\ell=1$, 
the total complexity for one pass over the data is of order $\mathcal{O}(Nd)$,
matching the complexity scaling of first-order methods while retaining the benefits of second-order information.
We also provide a study of an averaged version of the algorithm,
which can improve practical performance at a minimal additional computational cost.
\\
The objective of our work is therefore to propose:
(i) online stochastic Newton estimators,
(ii) developed within a setting as general as possible,
(iii) with reduced computational complexity and mini-batch processing 
(possibly $O(Nd)$ operations for one pass over the data), and
(iv) with strong theoretical guarantees, in particular asymptotic efficiency.
\\
To the best of our knowledge, no previous work has successfully combined these four aspects simultaneously.  

\paragraph{Paper Organization.}
Section \ref{sec:framework} introduces the notations, the optimization framework, and the underlying assumptions.
Section \ref{sec:inverse_estimation} presents the algorithm for recursively estimating the inverse Hessian.
This estimator is then incorporated into the proposed mini-batch stochastic Newton algorithm
in Section \ref{sec:mSNA algorithm}, where we establish its theoretical convergence analysis.
Section \ref{sec:streaming} details the implementation of this algorithm in the streaming (mini-batch) setting
and analyzes its computational complexity.
Finally, Section \ref{sec:numerical_experiments} provides numerical experiments 
to illustrate the algorithm's performance.
Proofs are deferred to the Appendix.

\section{Framework}\label{sec:framework}

\paragraph{Notations.}
The gradient and Hessian operators are denoted by $\nabla $ and $\nabla^2 $, respectively.
We denote by $\mathcal M_d (\mathbb R)$ the set of square matrices of order $d$ with real coefficients,
and by $\mathcal S_d^{++}(\mathbb R)$ the set of symmetric positive definite matrices of order $d$.
Let $I_d$ be the identity matrix of order $d$.
On $\mathcal M_d(\mathbb R)$, we denote by $\norm{ \cdot }_F$ (resp. $\norm{ \cdot }_{op}$)
the Frobenius norm (resp. the operator norm), and by $\langle \cdot , \cdot \rangle_F$ the Frobenius inner product.
For a vector $v \in \mathbb R^d$, we denote by $\norm{v}$ its Euclidean norm.
For a matrix $A \in \mathcal M_d(\mathbb R)$, $A^T$ denotes its transpose and $\mathrm{Tr}(A)$ its trace.
We denote the convergence in distribution by $\xrightarrow[]{\mathcal L}$.

Instead of studying general step-sizes $(\alpha_n)_{n \geq 1}$ satisfying the Robbins-Monro conditions
$\sum_{n} \alpha_n = + \infty$ and $\sum_{n} \alpha_n^2 < + \infty$,
we restrict ourselves for simplification to step sizes of the form
$\alpha_n \coloneq c_{\alpha} / (n^{\alpha} + c_{\alpha}')$, 
with $c_{\alpha} > 0, c_{\alpha}' \geq 0$, $1/2 < \alpha \leq 1$,
and following \cite{carpentier_stochastic_2015} we call such a sequence 
a {$\sigma(\alpha, c_{\alpha}, c_{\alpha}')$-sequence}.

\paragraph{Problem setting}
Let $F: \mathbb{R}^{d} \longrightarrow \mathbb{R}$ be a convex and twice differentiable function.
Assuming its existence and uniqueness, we consider the problem of estimating the minimizer $\theta^*$ of $F$, that is,
$$ \theta^* \coloneq \arg \min_{\theta \in \mathbb R^d} F(\theta) ,$$
with a sequence of estimators $(\theta_n)_{n \geq 0}$.
At each iteration $n \geq 1$, we assume access to stochastic oracles of the gradient and the Hessian.
More precisely, given a point $\theta \in \R^d$, the stochastic oracles return random estimates
$g_n(\theta)$ (resp. $h_n(\theta)$) of the gradient $\nabla F(\theta)$ (resp. the Hessian $\nabla^2 F(\theta)$). 
In addition, we denote by $H:= \nabla^{2}F(\theta^{*})$ the Hessian at the minimizer.
Let $(\mathcal F_n)_{n\geq 0}$ be the natural filtration generated by the iterates $(\theta_n)_{n \geq 0}$ and the oracles 
$(g_n, h_n)_{n \geq 1}$.
\begin{enumerate}[label=\textbf{(A\arabic*)}, series=assumptions]
	\item \label{Assumption oracles unbiased} \emph{(Unbiased Estimates).}
		For all $n \geq 1$ and $\theta \in \R^d$:
		\begin{enumerate}[label=\alph*)]
			\item \label{Assumption:unbiased_g}
				$\E \left[ g_{n}(\theta) \mid \mathcal F_{n-1} \right] = \nabla F(\theta)$ \quad \text{a.s.}
			\item \label{Assumption:unbiased_h}
				$\E \left[ h_n (\theta) \mid \mathcal F_{n-1} \right] = \nabla^2 F(\theta)$ \quad \text{a.s.}
		\end{enumerate}
	\item \label{Assumption expected smoothness} \emph{(Growth Condition).}
		There exists $\mathcal{L}_g, \mathcal{L}_h, \sigma^2 \geq 0$ such that for all $n \geq 1$ and $\theta \in \R^d$:
		\begin{enumerate}[label=\alph*)] 
			\item \label{Assumption:expected_smoothness_gradient} 
				$\E \left[ \norm{g_{n}(\theta)}^2 \mid \mathcal F_{n-1} \right]
				\leq  \mathcal{L}_g \left( F(\theta) - F(\theta^*) \right) + \sigma^2$ 
				\quad \text{a.s.}
			\item \label{Assumption:expected_smoothness_hessian}
				$\E \left[ \norm{h_n(\theta)}_{op}^2 \mid \mathcal{F}_{n-1} \right] \leq \mathcal{L}_h$ 
				\quad \text{a.s.}
		\end{enumerate}
	\item \label{Assumption fct: hessian positive continuous} \emph{(Hessian at Minimizer).}
		The Hessian matrix $H \coloneq \nabla^2 F(\theta^*)$ at the minimizer is positive definite,
		and the mapping $\theta	\mapsto \nabla^2 F(\theta)$ is continuous at $\theta^*$.
	\item \label{Assumption oracle lyapunov} \emph{(Lyapunov Conditions).}
		There exist $q , q' > 2$ and $M, M'>0$ such that:
		\begin{enumerate}[label=\alph*)] 
			\item \label{Assumption:Lyapunov_g} 
				$ \displaystyle \sup_{\substack{n\geq 1 \\ \norm{\theta - \theta^*} < M}}
				\E \left[ \norm{g_n(\theta) - \nabla F(\theta)}^{q}  \mid \F_{n-1} \right]
				< + \infty $
				\quad \text{a.s.}
			\item \label{Assumption:Lyapunov_H} 
				$ \displaystyle \sup_{\substack{n\geq 1 \\ \norm{\theta - \theta^*} < M'}}
				\E \left[ \norm{h_n(\theta) - \nabla^2 F(\theta)}_{F}^{q'}  \mid \F_{n-1} \right]
				< + \infty $
				\quad \text{a.s.}
		\end{enumerate}
	\item \label{Assumption oracle variance limit} \emph{(Covariance Limit).}
		If $\left( \theta_n \right)_{n \geq 0}$ converges almost surely to $\theta^*$,
		there exists $\Gamma \in \mathcal S_d^{++}$ such that 
		$\mathrm{Cov} \left( g_n(\theta_{n-1}) \mid \mathcal F_{n-1} \right) \xrightarrow[n \to \infty]{\text{a.s.}} \Gamma$.
\end{enumerate}
  
These assumptions are standard in stochastic approximation and follow those in
\cite{pelletier_weak_1998} and \cite{pelletier_almost_1998} (see also \cite{leluc_asymptotic_2023}).
Assumption \ref{Assumption oracles unbiased} could be relaxed to accommodate a vanishing bias,
but for simplicity we do not consider this case; we refer to \cite{surendran2024non} for an analysis thereof.
Assumption \ref{Assumption expected smoothness} \ref{Assumption:expected_smoothness_gradient} 
is crucial for obtaining the strong consistency of stochastic gradient estimators.
It is related to the  \emph{expected smoothness} introduced by \cite{gower_general_2019} and,
as in \cite{leluc_asymptotic_2023}, permits changes in the distribution of the gradient oracles.
Assumption \ref{Assumption expected smoothness} \ref{Assumption:expected_smoothness_hessian} is more restrictive
and implies in particular that the Hessian of $F$ is uniformly bounded.
Assumption \ref{Assumption fct: hessian positive continuous} implies that the function $F$ is 
locally strongly convex around $\theta^*$ and justifies the use of Newton methods.
This assumption is crucial for attaining asymptotic efficiency with any gradient-based method 
(see, e.g., \cite{pelletier_asymptotic_2000,godichon-baggioni_online_2019} for averaged stochastic 
gradient algorithms and \cite{leluc_asymptotic_2023} for conditioned gradient algorithms).
Assumption \ref{Assumption oracle lyapunov} is a standard known Lyapunov condition,
required to hold only locally around the minimizer $\theta^*$.
Finally, Assumption \ref{Assumption oracle variance limit} is a standard condition required to establish
the asymptotic normality of gradient-based estimators (\cite{pelletier_weak_1998}).

\paragraph{In practice}
This framework, which makes assumptions on generic oracles,
is more general than the standard setting where $F$ is defined as the expectation of a known function $f$,
i.e., $F(\theta) = \E_{\xi} \left[ f(\xi, \theta) \right]$.
This generality permits the oracle distributions to change over time 
(a setting also considered in \cite{leluc_asymptotic_2023}).
Furthermore, this oracle-based framework facilitates the study of mini-batching strategies,
including those with varying batch-sizes or non-uniform sampling.

In the standard setting where $F(\theta) = \E_{\xi} \left[ f(\xi, \theta) \right]$ 
and i.i.d. samples $(\xi_n)_{n \geq 1}$ are available, the natural oracles are 
\[ g_n(\theta) = \nabla f(\xi_n, \theta) \quad \text{and} \quad h_n(\theta) = \nabla^2 f(\xi_n, \theta) . \]
In the online mini-batch (streaming) setting, data are processed in blocks of size $b>1$.
Letting $\{\xi_{n,i}\}_{i=1}^b$ denote the samples in the block at iteration $n$,
the oracles are the empirical averages:
\[ g_n(\theta) = \frac{1}{b} \sum_{i=1}^b \nabla f(\xi_{n,i}, \theta) \quad \text{and} \quad h_n(\theta) = \frac{1}{b} \sum_{i=1}^b \nabla^2 f(\xi_{n,i}, \theta) . \]
The analysis of this mini-batch setting is detailed in Section \ref{sec:streaming}.

These oracle assumptions are satisfied if corresponding assumptions hold for the function $f$.
These assumptions are given in Appendix \ref{ass::B}. 
For classical problems such as logistic regression and p-means, a verification of the assumptions can be found 
in the appendix A of \cite{godichon-baggioni_online_2025}.

\section{ Online Estimation of the Inverse of a Positive Definite Matrix }\label{sec:inverse_estimation}

In this section, we aim to estimate recursively the inverse of the matrix $H \coloneq \nabla^2 F(\theta^*)$ 
assumed to be positive definite, 
with the help of a sequence of estimates $(H_n)_{n\geq 1}$ of $H$ adapted to the filtration $(\mathcal{F}_n)_{n \geq 0}$.
Within our framework, these estimates are obtained by evaluating the Hessian oracle $h_n$
at an estimator $\hat \theta_{n-1}$ that converges to $\theta^*$.
This sequence $(\hat \theta_n)_{n \geq 0}$ may represent the primary iterates $(\theta_n)_{n \geq 0}$
themselves, or an averaged version $(\bar \theta_n)_{n \geq 0}$ derived from them.

\subsection{Stochastic Gradient Estimation of the Inverse}

We first observe that $H^{-1}$ is the unique minimizer of the quadratic functional
$J : \mathcal{M}_{d} (\R ) \longrightarrow \R_{+}$
defined for all $A \in \mathcal{M}_{d} (\R)$ by:
    \[ J(A) = \norm{ H^{1/2} \left( A - H^{-1} \right)   }^2_F , \]
where $H^{1/2}$ denotes the unique symmetric positive definite square root of the matrix $H$.
The function $J$ is twice differentiable, 
and standard matrix calculus (see Appendix \ref{sec:gradient}) yields the gradient:
    \[ \nabla J(A) = 2 (H^{1/2})^T H^{1/2} (A - H^{-1})  = 2 \left( H A - I_d \right) . \]
Thus, the functional $J$ is $2 \lambda_{\min}(H)$-strongly convex and $2 \lambda_{\max}(H)$-smooth.

The main interest of the functional $J$ is that its gradient only depends on $H$ through the product $H A$.
One could then recursively estimate $H^{-1}$ using a SGD on the function $J$, 
by replacing $H$ by its estimates $H_n$.
The algorithm we will introduce in the next section is derived from a symmetric, 
positive definite factorization.
To aid in its analysis, we also define the related symmetrized functional:
\[ 
    J_{sym}(A) \coloneq \frac{1}{2} \left( J(A) + J(A^T) \right) , \]
    which has gradient 
\[
    \nabla J_{sym}(A) 
    = \frac{1}{2} \left( \nabla J(A) + \left( \nabla J (A^T) \right)^T \right) 
    = H A +  A H - 2 I_d ,
\]
which is symmetric when $A$ is symmetric.
A naive SGD estimator $(A_n)_{n \geq 0}$ based on either $J$ or $J_{sym}$ faces two significant limitations. 
First, a primary drawback of such an estimator is the computational cost. 
Indeed, at each step, the dense matrix multiplication requires $\mathcal{O}(d^{3})$ operations using standard algorithms.
To reduce this cost, we will consider a sketched version of the gradient.
A second limitation is that these updates do not guarantee the quadratic form associated to $A_n$ 
remains positive definite, i.e., $v^T A_n v > 0$ for all $v \in \R^d \setminus \{0\}$,
which is crucial to ensure descent directions when $A_n$ is used as a preconditioner.
We will thus propose a modification of the SGD update to ensure the positive definiteness of the estimator.

\subsection{A Reduced-Cost Positive Estimator of the Inverse}

Let $\ell \in \{1, \dots, d \}$ be an integer.
Let $(M_{\ell,n})_{n \geq 1}$ be a sequence of random diagonal projection matrices,
adapted to the filtration $(\mathcal{F}_n)_{n \geq 0}$.
At each step $n \geq 1$, $M_{\ell,n}$ is constructed by selecting a subset $I_n \subset \{1,\ldots,d\}$ of size $\ell$ uniformly at random, and setting $(M_{\ell,n})_{ii} = 1$ if $i \in I_n$ and $0$ otherwise.
We assume this random selection at step $n$ is independent of $\mathcal F_{n-1}$ and of the stochastic Hessian $H_n$.
It follows directly that $\E[M_{\ell,n} \mid \mathcal F_{n-1}] = \frac{\ell}{d} I_d$.
Let $\tilde H_n \coloneq M_{\ell,n} H_n$ be the sketched random Hessian.
Starting from a positive definite initial estimator $A_0 \coloneq I_d$,
we propose to estimate $H^{-1}$ by a sequence of estimators $(A_n)_{n \geq 0}$ defined recursively for all $n \geq 1$ by:
\begin{equation} \label{eq:def A_n}
    A_n \coloneq A_{n-1} 
    - \mathbf{1}_{\gamma_n \norm{ \tilde H_n }_{op} \leq \frac{1}{2}}
    \left(
        \gamma_{n} \left(  \tilde H_n A_{n-1} + A_{n-1} \tilde H_n^T - 2 M_{\ell,n} \right) 
        - \gamma_n^2 \tilde H_{n} A_{n-1} \tilde H_{n}^T
    \right) ,
\end{equation}
where $(\gamma_n)_{n \geq 1}$ is a $\sigma(\gamma, c_{\gamma}, c_{\gamma}')$-sequence with $\gamma \in \left(\frac{1}{2},1\right)$. 

The update \eqref{eq:def A_n} warrants several remarks.
First, the estimator $A_n$ is symmetric and positive definite by construction.
As $A_0$ is positive definite, the sequence $(A_n)_{n \geq 0}$ remains positive definite by induction, 
as shown by the following factorization:
\begin{equation} \label{eq:factorization An}
    A_{n} = 
    \left( I_d - \gamma_n \mathbf{1}_{\gamma_n \norm{ \tilde H_n }_{op} \leq \frac{1}{2}} \tilde H_n \right)
    A_{n-1} 
    \left( I_d - \gamma_n \mathbf{1}_{\gamma_n \norm{ \tilde H_n }_{op} \leq \frac{1}{2}} \tilde H_n \right)^T
    + 2 \gamma_n \mathbf{1}_{\gamma_n \norm{ \tilde H_n }_{op} \leq \frac{1}{2}} M_{\ell,n} .
\end{equation}
Indeed, $\left( I_d - \gamma_n \mathbf{1}_{\gamma_n \norm{ \tilde H_n }_{op} \leq \frac{1}{2}} \tilde H_n \right)$
is invertible, so by Sylvester's law of inertia, the first term is positive definite if $A_{n-1}$ is positive definite,
and the second term is positive semi-definite since $M_{\ell,n}$ is an orthogonal projection matrix.
Second, the $\mathcal{O}(\gamma_n)$ term in the update,
$\left(\tilde H_n A_{n-1} + A_{n-1} \tilde H_n^T - 2 M_{\ell,n} \right)$, 
is a symmetrisation of the projected stochastic gradient $2 M_{\ell,n} \left( H_n A_{n-1} - I_d \right)$
of $J$ at point $A_{n-1}$, using the fact that $A_{n-1}$ is symmetric.
Its expectation conditionally to $\mathcal{F}_{n-1}$ is $\frac{\ell}{d} \nabla J_{sym}(A_{n-1})$.
The sketching introduces an expected scaling factor of $\ell/d$, 
which slows down the convergence of the estimator to a neighborhood of $H^{-1}$.
Third, the $\mathcal{O}(\gamma_n^2)$ term in the update,
$- \gamma_n^2 \tilde H_{n} A_{n-1} \tilde H_{n}^T$,
is precisely the correction required to complete the quadratic form in the factorization above, 
which is what guarantees positive definiteness. 
Finally, the indicator $\mathbf{1}_{\gamma_n \norm{ \tilde H_n }_{op} \leq \frac{1}{2}}$ 
is a standard truncation to ensure the stability of the recursion by bounding the effect of large stochastic samples.

\subsection{Convergence Results}

We now state the convergence properties of the estimator $(A_n)_{n \geq 0}$.

\begin{proposition}\label{prop::cv rate of A_n}
    Let $(\hat \theta_n)_{n \geq 0}$ be a sequence of estimators of $\theta^*$ 
    adapted to the filtration $(\mathcal{F}_n)_{n \geq 0}$,
    and let $(A_n)_{n \geq 0}$ be the sequence defined by \eqref{eq:def A_n} 
    using $H_n \coloneq h_n(\hat \theta_{n-1})$.
    Suppose Assumptions \ref{Assumption oracles unbiased}, 
    \ref{Assumption expected smoothness}, \ref{Assumption:expected_smoothness_hessian},
    \ref{Assumption fct: hessian positive continuous},
    and \ref{Assumption oracle lyapunov}\ref{Assumption:Lyapunov_H} hold with $q' > \frac{2}{\gamma}$.

    If $\hat \theta_n \to \theta^*$ almost surely, then $A_n \to H^{-1}$ almost surely.

    In that case, let $\delta_n \coloneq \sup_{k \geq n} \norm{\hat \theta_{k-1} - \theta^*}$. 
    Let  $\mu : \mathbb{R}_+ \rightarrow \left[ 0; +\infty \right]$ be the local modulus of continuity of $\nabla^2 F$ 
    at $\theta^*$ defined by
    $
        \mu(r) \coloneq \sup_{\norm{\theta - \theta^*} \leq r}  \norm{\nabla^2 F(\theta) - \nabla^2 F(\theta^*)}_F .
    $
    By Assumption \ref{Assumption fct: hessian positive continuous}, the mapping $\mu$ is finite
    in a neighborhood of $0$ and satisfies $\lim_{r \to 0} \mu(r) = 0$.
    If $\gamma < 1$, then for any $\eta > 0$:
    \[
        \norm{A_n - H^{-1} }_F^2 
        = \mathcal{O} \left( \gamma_n (\ln{n})^{1 + \eta} 
        + \mu \left( \delta_{\lceil \frac{n}{2} \rceil} \right)^2  \right) \text{ a.s.}
    \]
\end{proposition}

Proposition \ref{prop::cv rate of A_n} establishes that the estimator sequence $(A_n)_{n \geq 0}$
is strongly consistent, provided that $\hat \theta_n \to \theta^*$ a.s.
Furthermore, it quantifies the asymptotic error by decomposing it into two components:
a variance term, $\mathcal{O}(\gamma_n (\ln{n})^{1 + \eta})$, which is inherent to the stochastic approximation,
and a bias term, $\mathcal{O}(\mu(\delta_{\lceil \frac{n}{2} \rceil})^2)$, 
which is controlled by the convergence rate of the estimator $\hat \theta_n$.
This sequence $(A_n)_{n \geq 0}$ will serve as the conditioning matrix 
in the stochastic Newton algorithm developed in the next section.

\subsection{Relation to Previous Works}

In \cite{godichon-baggioni_online_2025}, the authors proposed a similar online algorithm to estimate $H^{-1}$
based on a Robbins-Monro procedure to find a zero of the mapping $A \mapsto H A + A H - 2 I_d$, 
which is the gradient of the functional $J_{sym}$ defined above, and incorporating sketching.
Their algorithm ensures the positive definiteness of the estimator at each step
by introducing a more strict truncation than ours, and a projection onto a ball of matrices with 
slowly growing radius.

In the offline setting, \cite{agarwal_second-order_2017} proposed a algorithm 
to estimate the inverse of a positive definite matrix
based on a Taylor expansion of the inverse around a known matrix.
The recursion they propose is exactly a SGD on the quadratic functional $J$ defined above 
with fixed step-size $\gamma_n = 1$,
without sketching and without the positive definite correction term,
and after a few recursion they compute only the product of this inverse hessian estimate with the stochastic gradient,
resulting in a $\mathcal{O}(S d^2)$ cost per iteration, with $S$ the number of recursions.
The inverse hessian estimate starts from scratch at each iteration of the main algorithm.

\subsection{Remark on a Weighted Averaged Version}

We conclude this section with a remark on weighted averaging.
To obtain a faster convergence rate, one can also consider an averaged estimator
(see, e.g.\ \cite{polyak_acceleration_1992}, \cite{pelletier_asymptotic_2000}).
Since standard averaging can be sensitive to initialization,
we consider a weighted averaged version
(see \cite{mokkadem_generalization_2011} or \cite{boyer_asymptotic_2023}) 
recursively defined by $\bar A_0 \coloneq A_0$ and, for all $n \geq 1$:
\begin{align}\label{eq:def A_n,tau}
    \bar A_{n} = \left( 1 - \frac{\omega_n}{\sum_{k=0}^n \omega_k} \right) \bar A_{n-1}
    + \frac{\omega_n}{\sum_{k=0}^n \omega_k} A_n ,
    \quad \text{where} \quad \omega_n = ( \ln(n+1) )^{\tau} ,
\end{align}
where $\tau \geq 0$ is a weighting parameter; $\tau = 2$ is a common choice.
As $\bar A_n$ is a convex combination of the positive definite estimators $(A_k)_{0 \leq k \leq n}$,
it remains positive definite.

Despite the known benefits of averaging for optimizing asymptotic variance, the estimator $(\bar A_n)_{n \geq 0}$
will not be employed in the sequel.
The primary challenge for the estimator $(A_n)_{n \geq 0}$ is not the stochastic noise around the optimum,
but rather the slow convergence during the transient phase since the algorithm modifies only $\ell$ rows and columns at each iteration.
Averaging introduces inertia, which would further impede this initial convergence.
Furthermore, the averaging update \eqref{eq:def A_n,tau} is a dense operation,
requiring $\mathcal{O}(d^2)$ memory access to read and write all components of the estimator.
This dense access requirements contrasts with the non-averaged update \eqref{eq:def A_n};
while the calculation of that update term requires $\mathcal{O}(\ell d^2)$ operations,
the update itself only modifies the $\mathcal{O}(\ell d)$ components of $A_{n-1}$
corresponding to the selected $\ell$ rows and columns.
Given that the averaging step introduces both undesirable inertia and a dense memory access requirement,
we will retain the non-averaged estimator $(A_n)_{n \geq 0}$ for the stochastic Newton algorithm in the next section.
For a convergence analysis of a related weighted-averaging scheme 
for estimating the inverse of a positive definite matrix,
one can refer to \cite{godichon-baggioni_online_2025}.

\section{Algorithms}\label{sec:mSNA algorithm}

\subsection{A New Stochastic Newton Algorithm}
 
We propose here an algorithm named "masked Stochastic Newton algorithm" (mSNA) defined recursively for $n \geq 1$ by
\begin{align}
    \label{def::theta_n USNA}
    \theta_n & \coloneq \theta_{n-1} - \alpha_n \left( A_{n-1} + \nu_n I_d \right) g_{n}(\theta_{n-1}) \\
    \label{eq::def H_n USNA}
    \tilde H_n & \coloneq M_{\ell,n} h_n(\theta_{n-1}) \\
    \label{eq::def A_n USNA}
    A_{n}                & \coloneq A_{n-1} 
    - \mathbf{1}_{\gamma_n \norm{ \tilde H_n }_{op} \leq \frac{1}{2}}
    \left(
        \gamma_{n} \left(  \tilde H_n A_{n-1} + A_{n-1} \tilde H_n^T - 2 M_{\ell,n} \right) 
        - \gamma_n^2 \tilde H_{n} A_{n-1} \tilde H_{n}^T
    \right) ,
\end{align}
where $\theta_0 \in \mathbb{R}^d$ and $A_0 \in \mathcal{M}_d(\mathbb{R})$ are chosen arbitrarily,
$(M_{\ell,n})_{n \geq 1}$ is the sequence of random projection matrices defined in Section \ref{sec:inverse_estimation},
$(\alpha_n)_{n \geq 1}$ is a $\sigma(\alpha, c_{\alpha}, c_{\alpha}')$-sequence, 
$(\gamma_n)_{n \geq 1}$ is a $\sigma(\gamma, c_{\gamma}, c_{\gamma}')$-sequence
with $\gamma \in \left(\frac{1}{2},1\right)$,
and $(\nu_n)_{n \geq 1}$ decreases to $0$ and satisfies $\sum_{n\geq 1} \alpha_n \nu_n = +\infty$.
For that purpose, with any $\nu>0$, we take
 $\nu_{n}= \frac{\nu}{n^{1-\alpha}}$ if $\alpha < 1 $ and       $ \nu_{n} =  \frac{\nu}{\ln n}$ if $ \alpha = 1$.
The term $\nu_n$ is crucial theoretically to lower bound the smallest eigenvalue of the conditioning matrix, 
but in practice $\nu$ can be chosen arbitrarily small.
The following theorem gives the asymptotic properties of this stochastic Newton algorithm.

\begin{theorem}
\label{thm::USNA}
    Under Assumptions \ref{Assumption oracles unbiased} and \ref{Assumption expected smoothness}, 
    if $\alpha + \gamma > \frac{3}{2}$, 
    the estimator $(\theta_n)_{n \geq 0}$ defined by \eqref{def::theta_n USNA}  converge almost surely to $\theta^{*}$.     Assume also \ref{Assumption fct: hessian positive continuous}, \ref{Assumption oracle lyapunov} with  $q' > \frac{2}{\gamma}$,
    and \ref{Assumption oracle variance limit}.
    Then, denoting $\zeta \coloneq \mathbf 1_{\alpha=1} / 2 c_{\alpha}$,
    if $1 - \zeta > 0$ it holds that
    \[
        A_n \xrightarrow[n\to + \infty]{a.s.} H^{-1} \quad \text{and} \quad
        \frac{1}{\sqrt{\alpha_n}} \left( \theta_{n} - \theta^* \right)
        \xrightarrow[n \to +\infty]{\mathcal L} \mathcal N \left( 0, \frac{1}{2(1- \zeta)} H^{-1} \Gamma H^{-1} \right).
    \]
\end{theorem}
The proof is given in Appendix, and relies on a recent result of \cite{leluc_asymptotic_2023}
that ensures the asymptotic normality of conditioned SGD as soon as the conditioning matrix is strongly consistent.
The results of Theorem \ref{thm::USNA} present a main difference with \cite{godichon-baggioni_online_2025}
since we are able to obtain the convergence and asymptotic efficiency of the algorithm without an averaging step on $\theta_{n}$. 
This is given by the following corollary: 
\begin{corollary}\label{cor::efficiency USNA}
    Under the assumptions of Theorem \ref{thm::USNA}, if $\alpha_n = 1 / (n + c_{\alpha}')$, then
    \[
        \sqrt{n} \left( \theta_{n} - \theta^* \right)
        \xrightarrow[n \to +\infty]{\mathcal L} \mathcal N \left( 0, H^{-1} \Gamma H^{-1} \right).
    \]
\end{corollary}

In order to establish almost sure asymptotic rates of convergence using the results of 
\cite{boyer_asymptotic_2023} and \cite{godichon-baggioni_adaptive_2025}, we have to make a stronger assumption than \ref{Assumption fct: hessian positive continuous}.

\begin{enumerate}[label=\textbf{(A\arabic*)}, resume=assumptions]
	\item \label{Assumption fct: hessian lipschitz}  {\emph{(Lipschitz at $\theta^*$).} }
		The mapping $\theta \mapsto \nabla^2 F(\theta)$ is Lipschitz continuous at $\theta^*$,
        i.e. there exists $r>0$ and $L$ such that for all $\theta \in \mathcal{B}(\theta^*,r)$, 
        $ \norm{\nabla^2 F(\theta) - H}_F \leq L \norm{\theta - \theta^*}$.
\end{enumerate}
Observe that, as far as we know, averaged conditioned algorithms have only been studied in the literature 
with the additional Assumption \ref{Assumption fct: hessian lipschitz} that the Hessian is 
 {Lipschitz continuous at the minimizer,}
contrary to the non-averaged mSNA with step size $\alpha_n = 1/(n + c_{\alpha}')$.

\begin{theorem}\label{thm::USNA convergence rates}
    Under the assumptions of Theorem \ref{thm::USNA} and Assumption \ref{Assumption fct: hessian lipschitz},
    if the exponent $q$ in Assumption \ref{Assumption oracle lyapunov} satisfies $q > \frac{2}{\alpha}$,
    then for any $\eta > 0$:
    \[
        \norm{\theta_n - \theta^* }^2 = \mathcal{O} \left( \frac{\ln{n}}{n^\alpha} \right) \text{ a.s.}
        \quad \text{and} \quad
        \norm{A_n - H^{-1}}_F^2 = \mathcal{O} \left( \frac{(\ln{n})^{1+\eta}}{n^{\min \{\alpha, \gamma \}}} \right) \text{ a.s.}
    \]
\end{theorem}
\subsection{A Weighted Averaged Version}

We have seen that asymptotic efficiency can only be directly achieved by choosing $\alpha = 1$ and $c_\alpha = 1$, 
which corresponds to using steps of size $1/n$. 
However, this leads to very small updates, which can be problematic in the case of poor initialization 
(see \cite{cenac_efficient_2020,boyer_asymptotic_2023}).
To address this issue, one can choose $\alpha < 1$ to allow larger step sizes, 
and then apply an averaging step to recover asymptotic efficiency.
This averaging step can be applied to a wide range of conditioned SGD algorithms, including the mSNA algorithm.
For any estimator $(\theta_n)_{n \geq 0}$, we define the averaged estimator $(\bar \theta_n)_{n \geq 0}$ 
as a weighted average of the iterates $\theta_k$ for $k \leq n$, with weights of the form $\omega_k \coloneq (\ln (k+1))^{\tau}$ 
for a chosen $\tau \geq 0$ \cite{boyer_asymptotic_2023}:
\begin{align}
    \label{eq::def theta_n bar}
    \bar \theta_0 \coloneq \theta_0, \quad \text{and for } n \geq 1, \quad
    \bar \theta_n \coloneq \frac{\sum_{k=0}^n \omega_k \theta_k}{\sum_{k=0}^n \omega_k}.
\end{align}
With the convention $0^0 = 1$, taking $\tau = 0$ corresponds to the simple average of the iterates, 
while taking $\tau > 0$ enables to give more weights to the last estimates $\theta_{n}$.
This averaging step can be computed in a recursive manner:
\begin{align*}
    \bar \theta_n & =
    \left( 1 - \frac{\omega_n }{\sum_{k=0}^n \omega_k} \right) \bar \theta_{n-1}
    + \frac{\omega_n}{\sum_{k=0}^n \omega_k } \theta_n .
\end{align*}


\begin{remark}
Observe that since the averaged estimator $\bar \theta_n$ is expected to converge faster than the iterates $\theta_n$,
it is possible to query the Hessian oracle in $\bar \theta_n$ instead of $\theta_n$,
which modifies equation \eqref{eq::def H_n USNA} into
\begin{align}
    \label{eq::def H_n bar UWASNA}
        \tilde H_n & \coloneq M_{\ell,n} h_n(\bar \theta_{n-1}) .
\end{align}
However, the computations done to obtain an Hessian-vector product often require the gradient as intermediate step,
and for the update of $\theta_n$ the stochastic gradient must be computed at $\theta_n$. 
Therefore, querying the stochastic oracle of the Hessian at $\bar \theta_n$ instead of $\theta_n$ can even be less efficient.
\end{remark}

The following Theorem gives the strong consistency of the averaged estimates
(with either update \eqref{eq::def H_n USNA} or \eqref{eq::def H_n bar UWASNA}), 
their almost sure rate of convergence as well as their asymptotic efficiency.

\begin{theorem}\label{thm::UWASNA}
    Under Assumptions \ref{Assumption oracles unbiased} and \ref{Assumption expected smoothness}, 
    if $\alpha + \gamma > \frac{3}{2}$,
    the estimates $\theta_n$ and $\bar \theta_n$ defined by \eqref{def::theta_n USNA},
    \eqref{eq::def H_n USNA} (or \eqref{eq::def H_n bar UWASNA}), 
    \eqref{eq::def A_n USNA} and \eqref{eq::def theta_n bar}
    satisfy
    \[ \theta_{n} \xrightarrow[n\to + \infty]{a.s.} \theta^{*} \quad \text{and} \quad
        \bar \theta_{n} \xrightarrow[n\to + \infty]{a.s.} \theta^{*} . \]
    Assume also \ref{Assumption fct: hessian positive continuous},
    \ref{Assumption oracle lyapunov} with $q > \frac{2}{\alpha}$ and $q' > \frac{2}{\gamma}$,
    \ref{Assumption fct: hessian lipschitz}, and $\alpha < 1$.
    Then for any $\eta > 0$:
    \begin{align*}
        \lVert \theta_n - \theta^* \rVert^2 & = \mathcal O \left( \frac{\ln{n}}{n^\alpha} \right) \quad \text{and} \quad
        \lVert \bar \theta_n - \theta^* \rVert^2 = \mathcal O \left( \frac{\ln{n}}{n} \right) \quad \text{a.s.,}\\
        \norm{A_n - H^{-1} }_F^2 &= \mathcal O \left( \frac{(\ln{n})^{1+\eta}}{ n^{\min \{ \gamma, \alpha \}} } \right) \text{ a.s.}
        \quad \text{and} \quad
        \norm{ \bar A_n - H^{-1} }_F^2 = \mathcal O \left( \frac{(\ln{n})^{1+\eta}}{n} \right) \quad \text{a.s.}
    \end{align*}
    
    Moreover, with Assumption \ref{Assumption oracle variance limit}, we have
    \begin{align*}
        \sqrt{n} \left( \bar \theta_{n} - \theta^* \right) \xrightarrow[n \to +\infty]{\mathcal L} \mathcal N \left( 0, H^{-1} \Gamma H^{-1} \right).
    \end{align*}
\end{theorem}

Notably, if $\alpha < 1$, the averaged estimator attains the asymptotic efficiency for any $c_{\alpha} > 0$.
It is then possible to slow down the decrease of the step-size $\alpha_n$, by taking a larger $c_{\alpha}$ 
balanced by a larger $c_{\alpha}'$, and still have an asymptotically efficient algorithm.

\section{Streaming Newton algorithms} \label{sec:streaming}

\subsection{The Algorithm}

In this section, we focus on a streaming version of the mSNA algorithm, 
with potentially $\mathcal O(Nd)$ total operations for one pass over the data.
To this end, we consider the case where the objective function is the expectation of a known function $f$ twice differentiable in its second argument $\theta$:
\[ F(\theta) = \E \left[ f(\xi, \theta) \right] , \]
where $\xi$ is a random variable with values in a measured space $\mathcal X$.
Following the idea presented in \citep{godichon-baggioni_adaptive_2025}, 
we assume from now that at each iteration $n$, 
we have access to $b$ new i.i.d copies of $\xi$ arriving in a block $\{ \xi_{n,1}, \ldots , \xi_{n,b} \}$.
The streaming mSNA algorithm (and its weighted averaged version) are then defined 
by using the following stochastic oracles within the updates \eqref{def::theta_n USNA} and \eqref{eq::def A_n USNA}:
\begin{align}
    \label{eq::streaming oracles}
    g_n( \theta) = \frac{1}{b} \sum_{i=1}^{b} \nabla f(\xi_{n,i}, \theta)                
    \quad \text{and} \quad
    h_n(\theta) = \frac{1}{b} \sum_{i=1}^{b} \nabla^2 f(\xi_{n,i}, \theta) .
\end{align}
Considering $N=nb$ as the total number of samples processed after $n$ iterations,
the number of iterations required to process $N$ samples is $N / b$.
This is a factor $b$ fewer than in the purely online ($b=1$) setting,
leading to a reduction in total computational complexity, as discussed below.
{
    Although the number of iterations is reduced by a factor of $b$,
    the covariance of the mini-batch oracles \eqref{eq::streaming oracles} is also reduced by a factor of $b$.
    Consequently, the limiting covariance matrix $\Gamma$ from Assumption \ref{Assumption oracle variance limit} 
    becomes $\Sigma / b$, where $\Sigma := \text{Cov} \left[ \nabla f \left( \xi , \theta^{*} \right) \right]$
    is the covariance matrix for a single sample defined in Assumption \ref{Assumption oracle variance limit prime}.
}
The estimates are still asymptotically efficient, which is given by the following corollary 
of Theorems \ref{thm::USNA} and \ref{thm::UWASNA}.

\begin{corollary}\label{cor::streaming SNA}
    Under Assumptions 
    \ref{Assumption oracles unbiased prime} to \ref{Assumption oracle variance limit prime} (given in Appendix \ref{ass::B}) made on the oracles defined by \eqref{eq::streaming oracles},
    and assuming that \ref{Assumption oracle lyapunov prime}\ref{Assumption:Lyapunov_H prime} holds with $q' > \frac{2}{\gamma}$,
    the estimator $\theta_n$ of the streaming mSNA algorithm with step-size $\alpha_n = 1 / (n + c_{\alpha}')$ 
    and mini-batch size $b$ is asymptotically efficient:
    \begin{align*}
        \sqrt{N} \left( \theta_{n} - \theta^* \right) 
        & \xrightarrow[n \to +\infty]{\mathcal L} 
        \mathcal N \left( 0, H^{-1} \Sigma H^{-1} \right),
    \end{align*}
    where $\Sigma := \text{Cov} \left[ \nabla f \left( \xi , \theta^{*} \right) \right]$.
    Moreover, with Assumption \ref{Assumption fct: hessian lipschitz}, the estimator $\bar \theta_n$ of the streaming averaged mSNA algorithm 
    with step-size exponent $\alpha < 1$ and with mini-batch size $b$
    and assuming that \ref{Assumption oracle lyapunov prime} holds with $q > \frac{2}{\alpha}$ and $q' > \frac{2}{\gamma}$
    is asymptotically efficient:
    \begin{align*}
        \sqrt{N} \left( \bar{\theta}_{n} - \theta^* \right) 
        & \xrightarrow[n \to +\infty]{\mathcal L} 
        \mathcal N \left( 0, H^{-1} \Sigma H^{-1} \right) .
    \end{align*}
\end{corollary}

Notably, the streaming mSNA algorithm with $\alpha_n \coloneq 1 /(n + c_{\alpha}')$, 
as well as the streaming averaged mSNA algorithm
with any $(\alpha, c_{\alpha}, c_{\alpha}')$-sequence step-size with $\alpha < 1$, 
attain both the asymptotic efficiency 
and possess the same asymptotic distribution as the purely online versions with batch-size $b=1$. 
Observe that in  the case of the streaming averaged mSNA, 
one can take $c_{\alpha}=b^{a}$ with $a \in [0,1]$ (usually $a = 1-\alpha$). 
This enables larger steps when the number of iterations decreases 
(see \citep{godichon-baggioni_non-asymptotic_2023} for more details).

\subsection{Discussion on the Computational Complexity}

The per-iteration workload of our algorithm can be summarized as follows:

\paragraph{Computing $\tilde{H}_{n}$.}
 {
The complexity of computing the selected columns of the Hessian oracle depends on the function $f$.
We assume that computing the product of the Hessian $\nabla^2 f(\xi_{n,i}, \theta_{n-1})$
with a single vector (a Hessian-vector product) can be done efficiently,
without forming the full matrix (see, e.g., \citep{pearlmutter_fast_1994}).
Computing the $\ell$ columns selected by $M_{\ell,n}$ for a single sample $\xi_{n,i}$
thus requires $\mathcal O(\ell d)$ operations in many standard cases.
Computing $\tilde{H}_{n}$ involves computing these $\ell$ columns for each of the $b$ samples and averaging them,
resulting in a complexity of order $\mathcal O(\ell b d)$. 
}

\paragraph{Calculating $\norm{ \tilde{H}_{n} }_{op}$.} 
The calculation of the operator norm of $\tilde{H}_{n}$ requires $\mathcal O(\ell^2 d)$ operations,
as only $\ell$ columns of $\tilde{H}_{n}$ are non-zero.
As long as $\ell \leq \sqrt{d}$, this is of order $\mathcal O(d^{2})$.

\paragraph{Updating $A_{n}$.} 
Calculating the products $A_{n-1} \tilde H_n^T$ and $\tilde H_n A_{n-1} \tilde H_n^T$ requires $\mathcal O(\ell d^2)$ operations.
The overall update for $A_n$ is thus dominated by these matrix multiplications,
requiring $\mathcal O(\ell d^2)$  operations.

\paragraph{Updating $\theta_{n}$.} 
This first requires $\mathcal O(bd)$ operations to calculate the gradient oracle $g_n(\theta_{n-1})$, 
which is an average over the mini-batch.
Computing the product $\left(A_{n-1} + \nu_n I_d \right) g_n(\theta_{n-1})$ 
then requires $\mathcal{O}(d^{2})$ operations.
The total cost is $\mathcal{O}(bd + d^{2})$.

\paragraph{Updating $\bar \theta_{n}$.} This operations requires $\mathcal O(d)$ operations 
(vector scaling and addition).

\paragraph{Total cost of an iteration of the algorithm}
Summing these costs, the total cost per iteration is:
\[
    \underbrace{\mathcal O \left( \ell b d + \ell d^2  \right)}_{\text{Update $A_{n}$}}
    + \underbrace{\mathcal O \left( b d + d^{2} \right)}_{\text{Update $\theta_{n}$}}
    + \underbrace{\mathcal  O \left( d \right)}_{\text{Update $\bar \theta_{n}$ (if used)}}
    = \mathcal{O} \left( \ell \left( b d + d^2 \right) \right)
\]
Therefore, as long as the batch-size $b$ is not larger than the dimension $d$, 
the cost of an iteration is of order $\mathcal O(\ell d^2)$.
In a streaming setting with a single pass over $N$ data points, 
the algorithm performs $N/b$ iterations, leading to a total complexity of $\mathcal O(\frac{N}{b} \ell d^2)$.
Hence, by choosing $b=d$ and $\ell =1$, the algorithm achieves a total complexity
for one pass of order $\mathcal O(Nd)$ operations.

\section{Numerical Experiments}\label{sec:numerical_experiments}

In this section, we empirically evaluate the performance of the proposed
masked Stochastic Newton Algorithm (mSNA) and its weighted averaged variant,
particularly focusing on the streaming setting.
We compare our methods against Stochastic Gradient Descent (SGD) and its averaged
variant on both synthetic and real-world datasets. 
Our implementation uses the PyTorch framework and the code is available at \url{https://github.com/guillaume-salle/SNA}.
All experiments were conducted on a commodity laptop, with computations performed
on a multi-core CPU using 6 threads for parallelism; no GPU acceleration was utilized.
Although our simulations have focused on classical examples of linear and logistic regression,
there are many other applications, see for instance \cite{godichon-baggioni_online_2025}.

\subsection{Experimental Setup}

We compare the performance of the following algorithms:

\begin{itemize}
    \item \textbf{SGD:} The standard stochastic gradient descent update defined by 
    \begin{equation}\label{eq::sgd}
        \theta_n = \theta_{n-1} - \alpha_n g_n(\theta_{n-1}).
    \end{equation}
    \item \textbf{Averaged SGD:} The weighted-average variant with estimator $\bar \theta_n$ defined by \eqref{eq::def theta_n bar}.
    \item \textbf{mSNA:} Our proposed masked Stochastic Newton Algorithm defined by \eqref{def::theta_n USNA}, \eqref{eq::def H_n USNA}, and \eqref{eq::def A_n USNA}.
    \item \textbf{Averaged mSNA:} The weighted-average variant of our method with estimator $\bar \theta_n$ defined by \eqref{eq::def theta_n bar}.
\end{itemize}

For all optimizers, we use a batch size equal $b$ to the dimension $d$ of $\theta$.

For non averaged methods SGD and mSNA, we use the standard step size $\alpha_n \coloneq 1 / (n + n_0)$.
For the averaged SGD and averaged mSNA methods, we use the larger step size 
$\alpha_n \coloneq d^{0.25} / (n^{0.75} + d^{0.25} * n_0)$, and $\tau = 2$ for the averaging step \eqref{eq::def theta_n bar}.
The hyper-parameter $n_0$ determines the initial step size $\alpha_1$,
and is linked to the expected smoothness constant $\mathcal{L}$ of the objective function. 
It is set to $n_0 = d$ for synthetic data, and tuned by line search for real data.

In addition, in the case of the mSNA and averaged mSNA algorithms,
the parameter $\ell$ determines the number of columns and lines of the conditioning matrix $A_n$ that are updated at each iteration.  
Three choices of $\ell$ are considered: $\ell =1, d^{0.25},\sqrt{d}$. 
We use the step size $\gamma_n = 1 / (n^{0.75} + n_0)$ for the conditioning matrix update 
\eqref{eq::def A_n USNA}, and the regularization parameter $\nu = 0$ for the $\theta_n$ update \eqref{def::theta_n USNA}.
Note that for averaged mSNA, this choice of $\gamma = 0.75$ is the edge case of the condition $\alpha + \gamma > 3/2$
of Theorem \ref{thm::UWASNA} for the convergence of $\theta_n$.

\subsection{Experiments with Synthetic Data} \label{ssec:synthetic_data}

To assess the performance of the different methods, 
we analyze the evolution of the quadratic error $\|\theta_n - \theta^*\|^2$ 
with respect to the sample size. 
We also examine the evolution of the squared Frobenius norm error $\|A_n - H^{-1}\|_F^2$ 
for the inverse Hessian estimates, as well as the total computational time of the different algorithms. 
Observe that for both methods, the initialization $\theta_{0}$ is randomly chosen on the unit sphere centered at $\theta^*$ 
and $A_{0}$ is set to $I_{d}$.

\subsubsection{Linear Regression} \label{sssec:linear_regression_synthetic}

We consider a linear regression problem where 
$F(\theta) = \frac{1}{2} \mathbb{E}[(Y - X^T\theta)^2]$, with $\theta \in \mathbb{R}^{d}$.
It is an interesting case since the Hessian is known and constant, given by $\nabla^2 F(\theta) = \mathbb{E}[XX^T]=: \Sigma_{X}$
for all $\theta \in \mathbb{R}^d$, and we can so calculate explicitly $H^{-1}$.
To create a challenging scenario, we consider the following model:
\begin{itemize}
    \item 
    $X$ follows a multivariate Gaussian distribution $\mathcal{N}(0, \Sigma_X)$. 
    The covariance matrix $\Sigma_X$ is constructed to be ill-conditioned. 
    Specifically, $\Sigma_X = U \Lambda U^T$, where $\Lambda$ is a diagonal matrix with eigenvalues spaced evenly 
    between $10^{-2}$ and $1$, and $U$ is a random orthogonal matrix, 
    yielding a condition number of $10^2$ for $\Sigma_X$. 
    This construction ensures an ill-conditioned covariance matrix with non-axis-aligned principal
    components, making the optimization problem challenging even for conditioned SGD methods that
    use a diagonal conditioning matrix.
    \item 
    $Y = X^T \theta^* + \epsilon$, 
    where the parameter     $\theta^*$ is randomly chosen from a standardized Gaussian distribution 
    and $\epsilon \sim \mathcal{N}(0,1)$.
\end{itemize}
We take $d = 1000$ and $N = 10^7$, 
in order to have a moderately large dimension while still having
a reasonable number of iterations of the algorithms with a batch size equal to $d$.

\begin{figure}[htbp]
    \centering
    \includegraphics[width=0.95\textwidth]{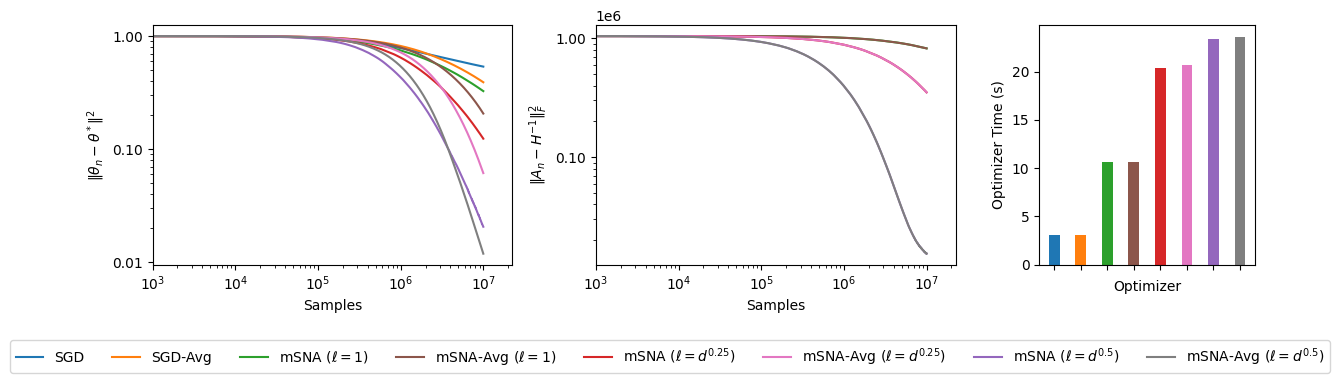}
    \caption{Linear regression with  $d=1000$. From the left to the right, comparison  of: evolution of the quadratic errors $\norm{ \theta_{n} - \theta^{*} }^{2}$,  quadratic errors $\norm{ A_{n} - H^{-1} }_{F}^{2}$, computational time for the different methods.   }
    \label{fig:linear_regression_synthetic}
\end{figure}

In Figure \ref{fig:linear_regression_synthetic}, 
we can see that, as expected in an ill-conditioned setting, 
standard gradient-based methods fail to converge effectively.
In contrast, Newton-type methods demonstrate significantly better performance, with only a slightly higher computational cost.
For all algorithmic variants considered, we observe that the averaged versions consistently outperform their non-averaged counterparts. 
As  anticipated, increasing the mask size $\ell$ leads to improved performance, albeit at the cost of increased computation time.

\subsubsection{Logistic Regression}

We consider here a logistic regression problem:
\begin{equation*}
    F(\theta) = \mathbb{E}[\log(1 + \exp(X^T \theta)) - X^T \theta Y], 
\end{equation*}
with $X \sim \mathcal{N}(0, \Sigma_X)$ as in Section \ref{sssec:linear_regression_synthetic}
and $Y |X \sim \mathcal{B}(\sigma(X^T \theta^*))$,
where $\sigma(x) = \frac{e^{x}}{1+e^{x}}$ 
is the sigmoid function.

For this case, we do not have a closed-form expression for the Hessian at the minimizer:
$H = \nabla^2 F(\theta^*) = \mathbb{E}[\sigma(X^T \theta^*) (1 - \sigma(X^T \theta^*)) XX^T]$.
Using a Monte-Carlo approach,
we compute an empirical estimate of $H$ on a large sample, 
and use it to compute an estimate of the inverse $H^{-1}$.

\begin{figure}[H]
    \centering
    \includegraphics[width=0.95\textwidth]{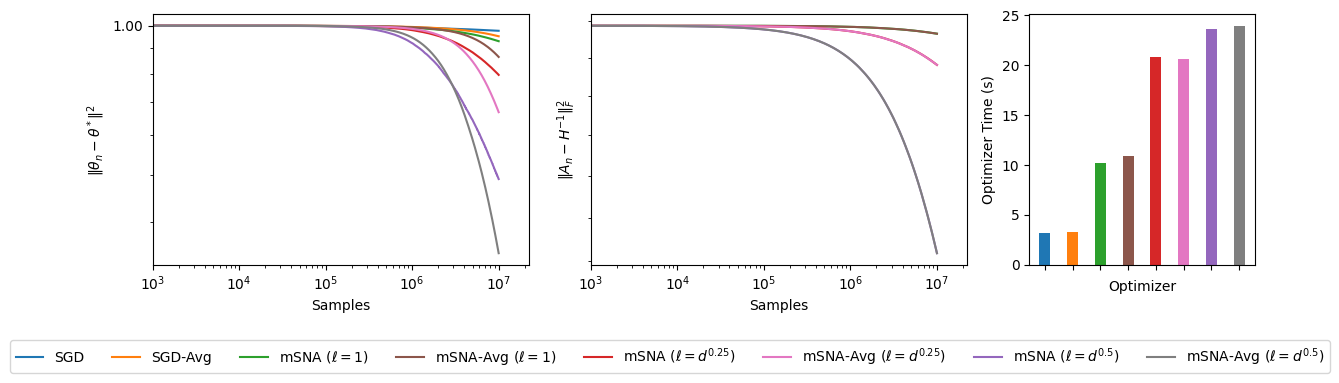}
    \caption{
        Logistic regression with  $d=1000$. 
        From the left to the right, comparison of: evolution of the quadratic errors $\| \theta_{n} - \theta^{*}\|^{2}$,  quadratic errors $\| A_{n} - H^{-1}\|_{F}^{2}$, computational time for the different methods.
    }
    \label{fig:logistic_regression_synthetic}
\end{figure}

In Figure \ref{fig:logistic_regression_synthetic}, 
we observe similar trends as in the linear regression setting. 
In the case of logistic regression, the Hessian $H$ exhibits very small
eigenvalues and a large conditioning number, which further diminish the performance of the SGD optimizer.

\subsection{Ridge logistic regression on Real Data}
\label{ssec:real_data}

We evaluate the algorithms on real-world datasets for binary classification using Ridge logistic regression. 
The characteristics of the studied datasets are summarized in Table \ref{tab:dataset_characteristics}. 

\begin{table}[h!]
\caption{Key characteristics of the datasets used in this study.}
\label{tab:dataset_characteristics}
\begin{tabular}{lrrrr}
\toprule
 & Dimension & Training Set Size & Init Set Size & Testing Set Size \\
\midrule
Adult & 98 & 38,682 & 390 & 9,769 \\
Connect-4 & 85 & 53,504 & 540 & 13,512 \\
Covtype & 55 & 460,160 & 4,648 & 116,203 \\
Mnist & 785 & 54,430 & 1,570 & 14,000 \\
Mushrooms & 95 & 6,309 & 190 & 1,625 \\
Phishing & 39 & 8,755 & 88 & 2,211 \\
\bottomrule
\end{tabular}
\end{table}

Each  data set is split into training and test sets. 
To evaluate the performance of the different methods, 
we consider both the prediction accuracy and the empirical loss function
evaluated both on the train set and the test set.

Adult \citep{adult_2} contains census data for income
prediction and includes multiple categorical variables converted into binary indicators. 
Connect-4 \citep{connect-4_26} is a dataset for the Connect-4 game, 
where the goal is to predict the next move based on the current board state.
COVTYPE \citep{covertype_31} originally includes multiple forest cover
types; in this study, we focus on distinguishing “Spruce/Fir” (labeled as 1)
from all other categories (labeled as 0). 
MNIST \citep{lecun_gradient-based_1998} is a dataset for handwritten digit recognition, 
where we focus on distinguishing the digit '$0$' (labeled as $1$) from all other digits (labeled as $0$).
Mushrooms \citep{mushroom_73} consists of morphological attributes of mushrooms used to determine their edibility. 
Phishing \citep{phishing_websites_327} is designed to detect
malicious websites and consists entirely of categorical features, which we
encode in binary form. 
These datasets are frequently adopted to benchmark binary classifiers 
\citep{toulis_asymptotic_2017}.

Our objective is to apply a ridge regression model to predict the binary response variable for each  dataset.
In the ridge regression, the objective function is defined by
\begin{equation*}
    F(\theta) =  \E \left[  \log(1 + \exp(X^T \theta)) - X^T \theta Y + \frac{\lambda}{2} \norm{\theta}^2 \right]
\end{equation*}
where $Y \in \{0, 1\}$, $X \in \R^{d}$  and $\lambda > 0$ is an $L_2$-regularization parameter. 
For all optimizers, we perform an initialization for the initial estimator $\theta_0$ by doing 
$100$ steps of gradient descent over an initial batch of $\max \lbrace N/100, 2d \rbrace$ data, 
starting at $0$ and using a constant learning rate found by line search. 
We do not include this initialization of $\theta_0$ in the computational time of the algorithms, 
since it is shared by all optimizers.

For the mSNA and averaged mSNA optimizers, we also perform an initialization for
the conditioning matrix $A_0$ by computing on the same batch an estimator of the hessian of $F$ at $\theta_0$ 
obtained before, and we inverse it. 
This initialization has a computational cost of $O(d^3)$, and is included in the computational time.

\begin{table}[!htbp]
\caption{Performance of streaming optimizers on various datasets.}
\label{tab:performance_results}
\scriptsize
\begin{tabular}{llrrrrr}
\toprule
Dataset & Optimizer & Train Acc & Test Acc & Train Loss & Test Loss & Time \\
\midrule
\multirow[c]{8}{*}{adult} & SGD-Avg & \bfseries 76.86 & \bfseries 76.87 & 5.72e-01 & 5.75e-01 & 31.8 ms \\
 & SGD & 76.75 & 76.72 & \bfseries 5.67e-01 & \bfseries 5.70e-01 & \bfseries 27.0 ms \\
 & mSNA-Avg ($\ell=0.25$) & 76.17 & 76.17 & 5.74e-01 & 5.77e-01 & 151.4 ms \\
 & mSNA-Avg ($\ell=0.5$) & 76.17 & 76.17 & 5.74e-01 & 5.77e-01 & 149.8 ms \\
 & mSNA-Avg & 76.17 & 76.17 & 5.74e-01 & 5.77e-01 & 78.6 ms \\
 & mSNA ($\ell=0.25$) & 76.17 & 76.17 & 5.74e-01 & 5.77e-01 & 131.3 ms \\
 & mSNA ($\ell=0.5$) & 76.17 & 76.17 & 5.74e-01 & 5.77e-01 & 146.7 ms \\
 & mSNA  & 76.17 & 76.17 & 5.74e-01 & 5.77e-01 & 78.1 ms \\
\midrule
\multirow[c]{8}{*}{connect-4} & SGD-Avg & 77.04 & 76.72 & 4.68e-01 & 4.69e-01 & \bfseries 42.2 ms \\
 & SGD & 77.00 & 76.61 & 4.69e-01 & 4.70e-01 & 44.8 ms \\
 & mSNA-Avg ($\ell=0.25$) & \bfseries 79.26 & \bfseries 79.18 & 4.38e-01 & 4.39e-01 & 205.0 ms \\
 & mSNA-Avg ($\ell=0.5$) & 79.26 & 79.18 & \bfseries 4.38e-01 & \bfseries 4.39e-01 & 217.4 ms \\
 & mSNA-Avg  & 79.25 & 79.17 & 4.38e-01 & 4.39e-01 & 123.5 ms \\
 & mSNA ($\ell=0.25$) & 79.16 & 79.08 & 4.39e-01 & 4.41e-01 & 201.5 ms \\
 & mSNA ($\ell=0.5$) & 79.14 & 79.10 & 4.39e-01 & 4.41e-01 & 217.8 ms \\
 & mSNA & 79.15 & 79.05 & 4.39e-01 & 4.41e-01 & 119.8 ms \\
\midrule
\multirow[c]{8}{*}{covtype} & SGD-Avg & \bfseries 58.36 & \bfseries 58.23 & \bfseries 6.66e-01 & \bfseries 6.66e-01 & 527.6 ms \\
 & SGD & 57.56 & 57.35 & 6.70e-01 & 6.71e-01 & \bfseries 494.9 ms \\
 & mSNA-Avg ($\ell=0.25$) & 55.64 & 55.68 & 6.79e-01 & 6.79e-01 & 2.35 s \\
 & mSNA-Avg ($\ell=0.5$) & 55.64 & 55.68 & 6.79e-01 & 6.79e-01 & 2.22 s \\
 & mSNA-Avg  & 55.64 & 55.68 & 6.79e-01 & 6.79e-01 & 1.51 s \\
 & mSNA ($\ell=0.25$) & 55.62 & 55.65 & 6.79e-01 & 6.79e-01 & 2.31 s \\
 & mSNA ($\ell=0.5$) & 55.62 & 55.65 & 6.79e-01 & 6.79e-01 & 2.19 s \\
 & mSNA & 55.62 & 55.65 & 6.79e-01 & 6.79e-01 & 1.47 s \\
\midrule
\multirow[c]{8}{*}{mnist} & SGD-Avg & 86.76 & 86.66 & 5.05e-01 & 5.21e-01 & 23.8 ms \\
 & SGD & \bfseries 86.95 & \bfseries 87.11 & \bfseries 3.09e-01 & \bfseries 3.14e-01 & \bfseries 19.0 ms \\
 & mSNA-Avg ($\ell=0.25$) & 86.08 & 85.76 & 3.28e-01 & 3.35e-01 & 102.3 ms \\
 & mSNA-Avg ($\ell=0.5$) & 86.08 & 85.76 & 3.28e-01 & 3.35e-01 & 98.8 ms \\
 & mSNA-Avg  & 86.08 & 85.80 & 3.28e-01 & 3.35e-01 & 65.0 ms \\
 & mSNA ($\ell=0.25$) & 85.75 & 85.49 & 3.34e-01 & 3.41e-01 & 243.2 ms \\
 & mSNA ($\ell=0.5$) & 85.75 & 85.49 & 3.34e-01 & 3.41e-01 & 133.7 ms \\
 & mSNA  & 85.76 & 85.49 & 3.35e-01 & 3.41e-01 & 50.2 ms \\
\midrule
\multirow[c]{8}{*}{mushrooms} & SGD-Avg & 94.37 & 95.08 & 2.43e-01 & 2.19e-01 & 9.3 ms \\
 & SGD & 94.37 & 95.08 & 2.41e-01 & 2.17e-01 & \bfseries 4.6 ms \\
 & mSNA-Avg ($\ell=0.25$) & 99.18 & 98.95 & 3.50e-02 & 4.38e-02 & 47.1 ms \\
 & mSNA-Avg ($\ell=0.5$) & 99.18 & 98.95 & 3.50e-02 & 4.37e-02 & 48.6 ms \\
 & mSNA-Avg & 99.18 & 98.95 & 3.50e-02 & 4.38e-02 & 27.8 ms \\
 & mSNA ($\ell=0.25$) & 99.69 & 99.08 & \bfseries 2.98e-02 & 3.83e-02 & 47.4 ms \\
 & mSNA ($\ell=0.5$) & \bfseries 99.70 & \bfseries 99.20 & 2.98e-02 & \bfseries 3.81e-02 & 49.3 ms \\
 & mSNA & 99.66 & 99.08 & 2.98e-02 & 3.84e-02 & 30.3 ms \\
\midrule
\multirow[c]{8}{*}{phishing} & SGD-Avg & 89.45 & 89.69 & 2.56e-01 & 2.46e-01 & 20.2 ms \\
 & SGD & 89.26 & 89.51 & 2.59e-01 & 2.49e-01 & \bfseries 13.3 ms \\
 & mSNA-Avg ($\ell=0.25$) & 93.79 & 94.08 & 1.61e-01 & 1.63e-01 & 95.0 ms \\
 & mSNA-Avg ($\ell=0.5$) & 93.77 & 94.08 & \bfseries 1.61e-01 & \bfseries 1.62e-01 & 124.2 ms \\
 & mSNA-Avg  & 93.75 & 94.08 & 1.62e-01 & 1.63e-01 & 52.5 ms \\
 & mSNA ($\ell=0.25$) & 93.77 & 94.12 & 1.66e-01 & 1.66e-01 & 78.8 ms \\
 & mSNA ($\ell=0.5$) & \bfseries 93.80 & \bfseries 94.17 & 1.65e-01 & 1.65e-01 & 95.4 ms \\
 & mSNA  & 93.77 & 94.08 & 1.67e-01 & 1.66e-01 & 62.7 ms \\
\bottomrule
\end{tabular}
\normalsize
\end{table}

In Table \ref{tab:performance_results}, 
we observe that the mSNA algorithms, along with their averaged versions have analogous behavior or can  outperform gradient-based methods, with higher but comparable computation times.
These results highlight not only the superior performance of our algorithms but also their practical readiness for deployment.

\section*{Conclusion} \label{sec:conclusion}

This paper introduced an efficient online mini-batch stochastic Newton algorithm (mSNA) designed for smooth convex optimization problems in stochastic settings.
Our work advances the Universal Stochastic Newton Algorithm (USNA) 
proposed in \citet{godichon-baggioni_online_2025} by establishing, 
in the case of local strong convexity around the minimizer, 
the asymptotic efficiency of our algorithm without requiring iterate averaging, 
a notable theoretical improvement.
This was made possible by proving a general theorem on the convergence of Robbins-Monro algorithms for linear functions, which is of independent interest.

The proposed mSNA algorithm operates in a streaming fashion, 
processing data in mini-batches, making it particularly well-suited for large-scale datasets. 
This enables our algorithm to achieve a total computational complexity of $O(Nd)$ for a single pass over $N$ data points.
Numerical experiments on both synthetic and real-world datasets for linear and logistic regression, with or without ridge penalization, demonstrate that mSNA and its averaged variant consistently achieve competitive performance compared to SGD and averaged SGD.
These results highlight not only the strong theoretical guarantees of our method but also its practical readiness for real-world problems.

\paragraph{Limitations and Future Work:} 
The current work focuses on the convex setting, and several avenues for
refinement and extension within this scope exist. 

First, formulating the problem of matrix inverse estimation as an
optimization problem opens the door to the possibility of estimating the inverse under
structural constraints, such as Ridge or Lasso regularization. 
This could be valuable when the matrix is sparse or when its smallest eigenvalue
is too close to zero.

Second, exploring the robustness of the algorithm under weaker conditions, such
as relaxing the smoothness assumptions (e.g., \ref{Assumption expected
smoothness} \ref{Assumption:expected_smoothness_hessian}), would be a valuable
direction for future work. 
Establishing the asymptotic efficiency
of the averaged mSNA without relying on the Hessian Lipschitz continuity
assumption \ref{Assumption fct: hessian lipschitz} remains an open and
interesting challenge.

Finally, while the random mask strategy used for Hessian approximation is both
simple and effective, investigating alternative sketching techniques or adaptive
mechanisms for selecting Hessian information could lead to further improvements.

\bibliographystyle{apalike}
\bibliography{bibliography.bib}

\appendix
\section{Proofs}\label{sec:proofs}

\subsection{Proof of Proposition \ref{prop::cv rate of A_n}}

\begin{proof}[Proof of Proposition \ref{prop::cv rate of A_n}]
    The following lemma bounds almost surely and asymptotically the operator norm of $A_n$, 
    without relying on information on the behavior of $\theta_n$.

    \begin{lemma}\label{lemma::biggest eigenvalue A_n}
        Under Assumptions \ref{Assumption oracles unbiased} and \ref{Assumption expected smoothness}\ref{Assumption:expected_smoothness_hessian}, for any $\eta > 0$, we have:
        \[ \lambda_{\max}(A_n) = \mathcal O(n^{1-\gamma} (\ln n)^{1+\eta}) \quad \text{a.s.} \]
    \end{lemma}
    The proof of this lemma is technical and is deferred to the next section \ref{sec::control largest eigenvalue An}.

    We have
    \begin{align*}
        A_{n} = A_{n-1} - \gamma_{n} \left( \nabla J_{sym}(A_{n-1}) + \epsilon_n + r_n \right) ,
    \end{align*}
    with a noise
    \begin{align*}
        \epsilon_n & \coloneq \left( 
            \tilde H_n A_{n-1} + A_{n-1} \tilde H_{n}^T - 2 M_{n,\ell} - \gamma_n \tilde H_n A_{n-1} \tilde H_{n}^T
        \right) \mathbf 1_{\gamma_n \norm{ \tilde H_{n} }_{op} \leq \frac{1}{2} } \\
            & \qquad - \E \left[ \left( \tilde H_n A_{n-1} + A_{n-1} \tilde H_{n}^T - 2 M_{n,\ell} 
                    - \gamma_n \tilde H_n A_{n-1} \tilde H_{n}^T \right)
            \mathbf 1_{ \gamma_n \norm{ \tilde H_{n} }_{op} \leq \frac{1}{2} } \mid \F_{n-1} \right] ,
    \end{align*}
    and residual terms
    \begin{align*}
        r_n       & \coloneq r^{(1)}_n + r^{(2)}_n + r^{(3)}_n                                          \\
        r^{(1)}_n & \coloneq
        A_{n-1} \left( \nabla^2 F(\theta_{n-1}) - H \right)
        + \left( \nabla^2 F(\theta_{n-1}) - H \right) A_{n-1}                                           \\
        r^{(2)}_n & \coloneq
        \E \left[ \gamma_{n} \tilde H_{n} A_{n-1} \tilde H_{n}^T
        \mathbf 1_{\gamma_n \norm{ \tilde H_n }_{op} \leq \frac{1}{2} } \mid \F_{n-1} \right] \\
        r^{(3)}_n & \coloneq
        - \E \left[ \left( \tilde H_n A_{n-1} + A_{n-1} \tilde H_n^T  + 2 M_{n,\ell} \right)
            \mathbf 1_{\gamma_n \norm{ \tilde H_{n} }_{op} > \frac{1}{2} } \mid \F_{n-1} \right] .
    \end{align*}
    The term $r^{(1)}$ accounts for the error of the estimator $\theta_n$, leading to a bias in the estimation of $H$.
    The term $r^{(2)}$ accounts for the term added in order to ensure the positivity of $A_n$,
    and the term $r^{(3)}$ accounts for the truncation.
    By construction, the noise $\epsilon_n$ is such that $\E \left[ \epsilon_n \mid \F_{n-1} \right] = 0$.
    Moreover, using inequality $(a + b)^p \leq 2^{p-1} (a^p + b^p)$ for $p \geq 2$, we have
    \begin{align*}
        \E \left[ \norm{\epsilon_n}_F^2 \mid \F_{n-1} \right]
            & \leq \E \left[
            \norm{\tilde H_n A_{n-1} + A_{n-1} \tilde H_{n}^T + 2 M_{n,\ell} 
            + \gamma_n \tilde H_n A_{n-1} \tilde H_{n}^T}_F^2
            \mathbf 1_{\gamma_n \norm{\tilde H_{n}}_{op} \leq \frac{1}{2} }
        \mid \F_{n-1} \right]                                    \\
            & \leq \E \left[ \left(
            2 \norm{ A_{n-1} }_{F} \norm{ \tilde H_n }_F
            + 2 \norm{ M_{n,\ell} }_F
            + \frac{1}{2} \norm{ A_{n-1} }_{F} \norm{ \tilde H_n }_F
        \right)^2 \mid \F_{n-1} \right]                          \\
            & \leq \E \left[
            2 \times \left( \frac{5}{2} \norm{ A_{n-1} }_{F} \norm{ \tilde H_n }_F \right)^2
            + 2 \times \left( 2 \norm{ M_{n,\ell} }_F \right)^2
        \mid \F_{n-1} \right]                                    \\
            & = \mathcal O \left( 1 + \norm{ A_{n-1} }_F^2 \right) .
    \end{align*}
    Since $\theta_n \xrightarrow[n \to \infty]{\text{a.s.}} \theta^*$ and
    using Assumption \ref{Assumption oracle lyapunov}\ref{Assumption:Lyapunov_H}, 
    for $n$ large enough $\norm{\epsilon_n}_F$ has a moment of order $q'$ and we have 
    with the same reasoning as order $2$:
    \[ 
        \E \left[ \norm{\epsilon_n}_F^{q'} \mid \F_{n-1} \right] 
        = \mathcal O \left( 1 + \norm{ A_{n-1} }_F^{q'} \right) .
    \]
    For the residual term $r_n$, taking a non-increasing modulus of continuity $\mu$ of $\nabla^2 F$ in $\theta^*$, we have
    \[ \norm{ \nabla^2 F(\theta_{n-1}) - H} \leq \mu ( \norm{ \theta_{n-1} - \theta^* } ) \leq \mu(\delta_n) . \]
    We get
    \begin{align*}
        \norm{ r^{(1)}_n }_F
            & \leq  2 \norm{ A_{n-1} }_{F} \norm{ \nabla^2 F(\theta_{n-1}) - H }_F
        = \mathcal{O} \left( \norm{ A_{n-1} }_F \mu(\delta_n) \right),                                                                     \\
        \norm{ r^{(2)}_n }_F
            & \leq \gamma_n \norm{ A_{n-1} }_{F} \E \left[ \norm{ \tilde H_n }_F^2 \mid \F_{n-1} \right]
        = \mathcal{O} \big( \gamma_n \norm{ A_{n-1} }_{F} \big) \text{ a.s.},                                                              \\
        \norm{ r^{(3)}_n }_F
            & \leq \E \left[ \left( 2 \norm{ A_{n-1} }_{F} \norm{ \tilde H_n }_F + 2 \norm{ M_{n,\ell} }_F \right)
        \times 2 \gamma_n \norm{ \tilde H_n }_F \mid \F_{n-1} \right]                                                    \\
            & = \mathcal O \left( \gamma_n \left( 1 + \norm{ A_{n-1} }_F \right) \right) \text{ a.s.}
    \end{align*}
    Therefore,
    \[
        \norm{ r_n }_F = \mathcal{O} \left( \left( 1 + \norm{ A_{n-1} }_F \right)  \left( \mu(\delta_n) + \gamma_n \right) \right) \text{ a.s.}
    \]
    By applying the Proposition \ref{thm::new_robbins_monro}
    with $a_n \coloneq \mu(\delta_n) + \gamma_n$ which is non-increasing,  we obtain for any $\eta > 0$:
    \begin{align*}
        \norm{ A_n - H^{-1} }_F^2 & = \mathcal O \left(
        \gamma_n (\ln{n})^{1+\eta} + \mu \left( \delta_{\lceil \frac{n}{2} \rceil } \right)^2 + \gamma_{\lceil \frac{n}{2} \rceil }^2
        \right) \quad \text{a.s.}                                                                                                                                            \\
                                        & = \mathcal O \left( \gamma_n (\ln{n})^{1+\eta} + \mu \left( \delta_{\lceil \frac{n}{2} \rceil } \right)^2 \right) \quad \text{a.s.}
    \end{align*}
    since $\gamma_{\lceil \frac{n}{2} \rceil }^2 = o(\gamma_n)$, which concludes this proof.
\end{proof}

\subsection{Proof of Theorem \ref{thm::USNA}}

\begin{proof}[Proof of Theorem \ref{thm::USNA}]
    We first prove the strong consistency of $(\theta_n)$.
    Recall that by induction on the factorization \eqref{eq:factorization An}, $A_n$ is positive definite for any $n \geq 0$.
    With the help of Lemma \ref{lemma::biggest eigenvalue A_n}, we have for any $\eta > 0$:
    \[ 
        \lambda_{\min}(A_n + \nu_n I_d) \geq \nu_n .
        \quad \text{and} \quad
        \lambda_{\max}(A_n + \nu_n I_d) = \mathcal O(n^{1-\gamma} (\ln n)^{1+\eta}) \quad \text{a.s.}
    \]
    By definition of $(\nu_n)$, and if $\alpha + \gamma > 3/2$, we have 
    \[ \sum_{n\geq1} \alpha_n \nu_n = +\infty \quad \text{and} \quad 
    \sum_{n\geq 1} \alpha_n n^{1-\gamma} (\ln n)^{1+\eta} < +\infty .\]
    Therefore, we can apply Theorem 1 from \cite{godichon-baggioni_adaptive_2025} 
    to obtain the strong consistency of $(\theta_n)$ towards $\theta^*$.
    Then, by applying Theorem \ref{thm::new_robbins_monro}, we obtain an almost sure rate of convergence of $(A_n)$ toward $H^{-1}$, 
    and in particular the strong consistency of $(A_n)$.
    Hence, if $1 - \zeta > 0$ we can apply Theorem 2 in \cite{leluc_asymptotic_2023} to obtain the asymptotic normality of $\theta_n$:
    \[
        \frac{1}{\sqrt{\alpha_n}} \left( \theta_n - \theta^* \right)
        \xrightarrow[n \to +\infty]{\mathcal L}
        \mathcal N \left( 0, \frac{1}{2(1-\zeta)} H^{-1} \Gamma H^{-1} \right) .
    \]
    To obtain the rate of almost sure convergence with Assumption \ref{Assumption fct: hessian lipschitz},
    from \cite{boyer_asymptotic_2023}, if $\alpha = 1$ we apply Theorem 3.2, and if $\alpha \in (2/b,1)$ we apply Theorem 4.2.
\end{proof}

\begin{proof}[Proof of Theorem \ref{thm::UWASNA}]
    The norm of the averaged matrices $\bar A_n$ still verify the upper bound of Lemma \ref{lemma::biggest eigenvalue A_n},
    since $\bar A_n$ is a convex combination of $A_j$ for $j \leq n$ and the upper bound is increasing. 
    Hence, we can follow the same reasoning as the proof of Theorem \ref{thm::USNA} to obtain the results on $\theta_n$ and $A_n$. 

    By application of the Toeplitz lemma, the strong consistency of $\theta_n$ (resp. $A_n$) gives the strong consistency of $\bar \theta_n$ $(resp. \bar A_n)$.
    To obtain the convergence rate and asymptotic normality of $\bar \theta_n$,
    we can apply Theorem 3 in \cite{godichon-baggioni_adaptive_2025}.
    Their Assumption 5 is implied by Assumptions \ref{Assumption fct: hessian lipschitz} and \ref{Assumption expected smoothness}\ref{Assumption:expected_smoothness_hessian}.
    Then, Theorem \ref{thm::new_robbins_monro} with Assumption \ref{Assumption fct: hessian lipschitz} gives us the convergence rate of $\bar A_n$ .
\end{proof}

\subsection{Control on the largest eigenvalue of \texorpdfstring{$A_n$}{An}} \label{sec::control largest eigenvalue An}

This proof of Lemma \ref{lemma::biggest eigenvalue A_n} follows closely the proof of Proposition 6.1 in \cite{godichon-baggioni_online_2025}.

\begin{proof}[Proof of Lemma \ref{lemma::biggest eigenvalue A_n}]
    Let $n \geq 1$.  We have:
    \begin{align}
        &\mathbb E\left[ \lVert A_{n} \rVert_F^2 \mid \mathcal F_{n-1} \right] 
        = \lVert A_{n-1} \rVert_F^2 + 2 \langle A_{n-1}, \mathbb E\left[ A_{n} - A_{n-1} \mid \mathcal F_{n-1} \right] \rangle_F 
        + \mathbb E\left[ \lVert A_{n} - A_{n-1} \rVert_F^2 \mid \mathcal F_{n-1} \right] . \label{decomp largest vp} 
    \end{align}
    For the last term of the right-hand side of \eqref{decomp largest vp}, we have using the truncation and then the inequality $(a + b)^2 \leq 2a^2 + 2b^2$:
    \begin{align*}
        &\mathbb E\left[ \left\lVert A_n - A_{n-1} \right\rVert^2_F \mid \mathcal F_{n-1} \right] \\
        &=  \mathbb E\left[ \left\lVert \gamma_{n} \left( \tilde H_{n} A_{n-1} + A_{n-1} \tilde H_{n}^T - 2 M_{n,\ell} \right) 
        - \gamma_{n}^2 \tilde H_{n} A_{n-1} \tilde H_{n}^T \right\rVert^2_F \, 
        \mathbf 1_{\gamma_n \lVert \tilde H_{n} \rVert_{op} \leq \frac{1}{2}} \mid \mathcal F_{n-1} \right] \\
        &\leq \gamma_n^2 \mathbb E\left[ \left( 2 \lVert \tilde H_{n} \rVert_{F} \lVert A_{n-1} \rVert_F + 2 \lVert M_{n,\ell} \rVert_F
        + \gamma_{n} \lVert \tilde H_{n} \rVert^2_{F} \lVert A_{n-1} \rVert_F \right)^2 
        \mathbf 1_{\lVert \gamma_n \tilde H_{n} \rVert_{op} \leq \frac{1}{2}} \mid \mathcal F_{n-1} \right] \\
        &\leq \gamma_n^2 \mathbb E\left[ \left( 2 \lVert \tilde H_{n} \rVert_{F} \lVert A_{n-1} \rVert_F 
        + 2 \lVert M_{n,\ell} \rVert_F + \frac{1}{2} \lVert \tilde H_{n} \rVert_{F} \lVert A_{n-1} \rVert_F \right)^2 \mid \mathcal F_{n-1} \right] \\
        &\leq \gamma_{n}^2 \left( 2 \times \frac{25}{4} \, \mathbb E\left[ \lVert \tilde H_{n} \rVert_{F}^2 \mid \mathcal F_{n-1} \right]
        \lVert A_{n-1} \rVert^2_F + 2 \times 4 \mathbb{E} \left[ \norm{ M_{n,\ell} }_F^2 \mid \mathcal{F}_{n-1} \right] \right) \\
        &= \mathcal{O} \left( \gamma_n^2 \left( \norm{ A_{n-1} }_F^2 + 1 \right) \right) \text{ a.s.}
    \end{align*}
    Regarding the inner product term of \eqref{decomp largest vp}, in order to form the expectation of $\tilde H_{n-1}$, we can write
    \begin{align*}
        &\mathbb E\left[ A_{n} - A_{n-1} \mid \mathcal F_{n-1} \right] \\
        &= \mathbb E\left[ \left( -\gamma_{n} \left( \tilde H_{n} A_{n-1} + A_{n-1} \tilde H_{n}^T - 2 M_{n,\ell} \right) 
        + \gamma_{n}^2 \tilde H_{n} A_{n-1} \tilde H_{n}^T \right)
        \mathbf 1_{\gamma_n \lVert \tilde H_{n} \rVert_{op} \leq \frac{1}{2} }  \mid \mathcal F_{n-1} \right] \\
        &= -\gamma_{n} \mathbb E\left[ \tilde H_{n} A_{n-1} + A_{n-1} \tilde H_{n}^T - 2 M_{n,\ell} \mid \mathcal F_{n-1} \right]
        + \gamma_{n}^2 \mathbb E\left[ \tilde H_{n} A_{n-1} \tilde H_{n}^T 
        \mathbf 1_{\gamma_n \lVert \tilde H_{n} \rVert_{op} \leq \frac{1}{2} } \mid \mathcal F_{n-1} \right] \\
        &\quad\, + \gamma_{n} \mathbb E\left[ \left( \tilde H_{n} A_{n-1} + A_{n-1} \tilde H_{n}^T - 2 M_{n,\ell} \right)  
        \mathbf 1_{\gamma_n \lVert \tilde H_{n} \rVert_{op} > \frac{1}{2} } \mid \mathcal F_{n-1} \right] . 
        \refstepcounter{equation}\tag{\theequation}\label{decomp 3 terms}
    \end{align*}
    We now bound the inner product of $A_{n-1}$ with each of the three terms of \eqref{decomp 3 terms}. \\
    By Assumption \ref{Assumption oracles unbiased}, the conditional expectation of $\tilde H_{n}$ is the hessian $\nabla^2 F(\theta_{n-1})$,
    which is positive since $F$ is convex. We get:
    \begin{align*}
        &\langle A_{n-1}, \nabla^2 F(\theta_{n-1}) A_{n-1} \rangle_F = \mathrm{Tr}\left( A_{n-1}^T \nabla^2 F(\theta_{n-1}) A_{n-1} \right) \geq 0 \\
        &\langle A_{n-1}, A_{n-1} \nabla^2 F(\theta_{n-1}) \rangle_F = \mathrm{Tr}\left( A_{n-1}^T A_{n-1} \nabla^2 F(\theta_{n-1}) \right)
        = \mathrm{Tr} \left( A_{n-1} \nabla^2 F(\theta_{n-1}) A_{n-1}^T \right) 
        \geq 0 .
    \end{align*}
    Thus, for any $\zeta_{n} > 0$ which we will choose later in order to balance each component, 
    we get the following bound for the inner product of $A_{n-1}$ with the first term of \eqref{decomp 3 terms}:
    \begin{align*}
        \langle A_{n-1}, -\gamma_{n} \mathbb E\left[ \tilde H_{n} A_{n-1} + A_{n-1} \tilde H_{n}^T - 2 M_{n,\ell} \mid \mathcal F_{n-1} \right] \rangle_F 
        &\leq 2\gamma_{n} \langle A_{n-1}, \mathbb{E} \left[ M_{n,\ell} \mid \mathcal{F}_{n-1} \right] \rangle_F \\
        &\leq \frac{\gamma_{n}}{\zeta_{n}} \lVert A_{n-1} \rVert^2_F + \frac{\ell^2}{d^2} \gamma_{n} \zeta_{n} .
    \end{align*}
    For the inner product of $A_{n-1}$ with the second term of \eqref{decomp 3 terms}, we get
    \begin{align*}
        \langle A_{n-1}, \gamma_{n}^2 \mathbb E\left[ \tilde H_{n} A_{n-1} \tilde H_{n}^T \mathbf 1_{\gamma_n \lVert \tilde H_{n} \rVert_{op} \leq \frac{1}{2} } 
        \mid \mathcal F_{n-1} \right] \rangle_F 
        &\leq \gamma_{n}^2 \lVert A_{n-1} \rVert^2 \mathbb E\left[ \lVert \tilde H_{n} \rVert^2_{F}  
        \mid \mathcal F_{n-1} \right] .
    \end{align*}
    For the inner product of $A_{n-1}$ with the third term of \eqref{decomp 3 terms}, we get:
    \begin{align*}
        &\langle A_{n-1}, \gamma_{n} \mathbb E\left[ \left( \tilde H_{n} A_{n-1} + A_{n-1} \tilde H_{n}^T - 2 M_{n,\ell} \right) 
        \mathbf 1_{\gamma_n \lVert \tilde H_{n} \rVert_{op} > \frac{1}{2} } \mid \mathcal F_{n-1} \right] \rangle_F \\
        &\leq 2\gamma_{n} \lVert A_{n-1} \rVert^2_F \mathbb E\left[ \lVert \tilde H_{n} \rVert_{F}  
        \mathbf 1_{ \gamma_n \lVert \tilde H_{n} \rVert_{op} > \frac{1}{2} } \mid \mathcal F_{n-1} \right] 
        + 2\gamma_{n} | \langle A_{n-1}, \mathbb{E} \left[ M_{n,\ell} \mid \mathcal{F}_{n-1} \right] \rangle_F | \\
        &\leq 2\gamma_{n} \lVert A_{n-1} \rVert^2_F \mathbb E\left[ 2 \gamma_n \lVert \tilde H_{n} \rVert_{F}^2  
        \mid \mathcal F_{n-1} \right]
        + \frac{\gamma_n }{\zeta_n} \lVert A_{n-1} \rVert_F^2 + \gamma_{n} \zeta_{n} d \\
        &= \mathcal{O} \left( \left( \gamma_n^2 + \frac{\gamma_n }{\zeta_n } \right) \lVert A_{n-1} \rVert^2_F + \gamma_n \zeta_n \right) \text{ a.s.}
    \end{align*}
    Finally, we obtain
    \begin{align*}
        \mathbb E\left[ \lVert A_{n} \rVert_F^2 \mid \mathcal F_{n-1} \right] \leq 
        \left( 1 + E_n \right) \lVert A_{n-1} \rVert^2_F + B_n,
    \end{align*}
    with $E_n = \mathcal{O}(\gamma_n^2 + \gamma_n / \zeta_n)$ and $B_n = \mathcal{O}(\gamma_n^2 + \gamma_n \zeta_n)$.
    We can now set $\zeta_n := n^{1-\gamma}(\ln n)^{1+\eta}$ with $\eta>0$, so that $(\frac{\gamma_n}{\zeta_n})$ is summable. 
    The \mbox{Lemma \ref{Robbins-Siegmund}} applied with $a_n := \zeta_n^2$ and $D_n := 0$ gives us: 
    \[ \lVert A_{n-1} \rVert^2_F = o(\zeta_{n-1}^2) \quad \text{a.s.} \]
\end{proof}
\section{Auxiliary results}

\begin{theorem}\label{thm::new_robbins_monro}
    Let $H \in \mathcal{M}_d(\mathbb{R})$ be a positive definite matrix and $z^* \in \mathbb{R}^d$.
    Consider the linear Robbins-Monro algorithm
    \[ z_n = z_{n-1} - \gamma_n \left( h(z_{n-1}) + \epsilon_n + r_n \right), \]
    where $h(z) = H(z - z^*)$, $(\epsilon_n)$ and $(r_n)$ are random sequences adapted to a filtration $(\mathcal{F}_n)$,
    $z_0$ is an $\mathcal{F}_0$-measurable $\mathbb{R}^d$-valued random variable,
    and $\gamma_n = c / (n^{\gamma} + c')$ with $\gamma \in \left(\frac{1}{2}, 1\right)$, $c > 0$, and $c' \geq 0$.
    Assume there exists an exponent $p \geq 0$ such that $\|z_n\| = \mathcal{O}(n^p)$ a.s.
    Suppose further that for all $n \geq 1$,
    \begin{align*}
        \E \left[ \epsilon_n \mid \F_{n-1} \right]       & = 0 \text{ a.s.}, \\
        \E \left[ \|\epsilon_n\|^b \mid \F_{n-1} \right] & = \mathcal{O} \left( \left(1 + \|z_{n-1}\| \right)^b \right) \text{ a.s.},\\
        \|r_n\| & = \mathcal{O} \left( a_n \left( 1 + \|z_{n-1}\| \right) \right) \text{ a.s.},
    \end{align*}
    for some $b > \frac{2}{\gamma}$, and a random sequence $(a_n)$ that almost surely is non-increasing and converging to $0$.
    Then, for any $\eta > 0$,
    \[ \|z_n - z^*\|^2 = \mathcal{O} \left( \gamma_n (\ln n)^{1+\eta} + a_{\lceil n/2 \rceil}^2 \right) \text{ a.s.} \]
\end{theorem}

\begin{proof}
    Define $\beta_{k,n} \coloneqq \prod_{j=k}^n (I_d - \gamma_j H)$ for all $k \leq n$, and $\beta_{k,n} \coloneqq I_d$ for $k>n$.
    By induction, we express
    \begin{equation}
        \label{eq::z_n induction}
        z_n - z^* = \underbrace{\beta_{1,n} (z_0 - z^*)}_{Z_n}
        + \underbrace{\sum_{k=1}^n \beta_{k+1,n} \gamma_k \epsilon_k}_{E_n}
        + \underbrace{\sum_{k=1}^n \beta_{k+1,n} \gamma_k r_k}_{R_n}.
    \end{equation}
    Choose $n_0 \in \mathbb{N}$ such that for all $n \geq n_0$,
    $\lambda_{\max}(H) \gamma_{n} \leq 1$. Since $H$ is positive, for $n \geq n_0$,
    \[
        \|I_d - \gamma_n H\|_{op}
        \leq \max \{ 1-\lambda_{\min}(H)\gamma_n , 1 - \lambda_{\max}(H) \gamma_n \}
        = 1 - \lambda_{\min}(H) \gamma_n.
    \]

    \paragraph{Bounding $Z_n$.} The term $Z_n$ tends exponentially fast to $0$. Indeed, for $n \geq n_0$,
    \[
        \|\beta_{1,n}\|_{op} \leq \|\beta_{1,n_0-1}\|_{op} \prod_{j=n_0}^n \left(1 - \lambda_{\min}(H) \gamma_j \right)
        \leq \|\beta_{1,n_0-1}\|_{op} \exp \left( -\lambda_{\min}(H) \sum_{j=n_0}^n \gamma_j \right),
    \]
    and $\sum_{j=n_0}^n \gamma_j \sim c n^{1-\gamma} / (1 - \gamma)$.

    \paragraph{Bounding $E_n$.}
    Let $p \geq 0$ be such that $\|z_{n-1}\| = \mathcal{O}(n^{p})$ a.s.
    Since $b > \frac{2}{\gamma}$, apply Theorem 6.1 and Remark 6.1 from \cite{cenac_efficient_2020}, setting
    $T_n \coloneqq (1 + \|z_{n-1}\|) I_d$ and ${\xi}_n \coloneqq \epsilon_n / (1 + \|z_{n-1}\|)$. Then,
    $\|T_n\|_{op} = o \left( n^{p} (\ln n)^{\eta/2} \right)$ a.s. for any $\eta > 0$.
    Although the theorem assumes non-negative exponents $a,b$ such that $\|T_n\|_{op} = o \left( (\ln n)^b / n^a \right)$ a.s.,
    it extends to negative values with the same proof, and we take $a \coloneqq -p$. This yields
    \[ \|E_n\|^2 = \mathcal{O} \left( n^{2p} \gamma_n (\ln n)^{1 + \eta} \right) \text{ a.s.} \]
    \paragraph{Bounding $R_n$.}
    We have for $n \geq n_0$,
    \[
        \|R_n\| \leq \|I_d - \gamma_n H\|_{op} \|R_{n-1}\| + \gamma_n \|r_n\| \leq \left( 1 - \lambda_{\min}(H) \gamma_n \right) \|R_{n-1}\| + \gamma_n \|r_{n}\|.
    \]
    Using equation \eqref{eq::z_n induction} and the hypothesis on $(r_n)$, we have
    \[
        \|r_n\| = \mathcal{O} \left( a_n \left( 1 + \|Z_{n-1}\| + \|E_{n-1}\| + \|R_{n-1}\| \right) \right)
        \text{ a.s.}
    \]
    We showed that $ 1 + \|Z_{n}\| + \|E_{n}\| = \mathcal{O} \left( \max \left\{ 1 , n^{p-\gamma/2} (\ln n)^{(1+\eta)/2} \right\} \right)$ a.s.
    Using the hypothesis on $(a_n)$, we can apply a.e. the deterministic result of Lemma \ref{lemma::bound Rn}.
    This yields with some $\kappa > 0$:
    \[
        \|R_n\| = \mathcal O \left( \left( \exp(-\kappa n^{1-\gamma}) +  a_{\lceil n/2 \rceil} \right) \max \left\{ 1 , n^{p-\gamma/2} (\ln n)^{(1+\eta)/2} \right\} \right) \text{ a.s.}
    \]
    We obtain with \eqref{eq::z_n induction}, after neglecting the exponential terms and using that $a_n = o(1)$ a.s.:
    \begin{equation}
        \label{eq::z_n rate}
        \|z_n - z^*\| = \mathcal O \left( a_{\lceil n/2 \rceil} + n^{p-\gamma/2} (\ln n)^{(1+\eta)/2} \right) \text{ a.s.}
    \end{equation}
    Let $k \in \mathbb N$ be the smallest integer such that $\|z_{n-1}\| = \mathcal{O}(n^{k\gamma/4})$ a.s. 
    Applying Equation \eqref{eq::z_n rate} with $p \coloneq k \gamma / 4$, 
    then using that $(\ln n)^{(1+\eta)/2} = o(n^{k\gamma/4})$ and $a_n = o(1)$ a.s., we get 
    \begin{align*}
        \|z_n\| \leq \|z_n - z^*\| + \|z^*\| 
        = \mathcal{O} \left( a_{\lceil n/2 \rceil} + n^{(k-2)\gamma/4}(\ln n)^{(1+\eta)/2}+ 1 \right)
        \subset \mathcal{O} \left( n^{(k-1)\gamma/4} + 1 \right) \text{ a.s.}
    \end{align*}
    If $k>0$, it leads to $\|z_n\| = \mathcal{O}(n^{(k-1)\gamma/4})$ a.s., which is a contradiction with the definition of $k$.
    Therefore $k=0$, and taking $p=0$ in \eqref{eq::z_n rate} gives us the desired rate of convergence.
\end{proof}

\begin{lemma}\label{lemma::bound Rn}
    Let $(r_n)_{n\geq0}, (s_n)_{n\geq 1}, (a_n)_{n\geq 1}$ be three non-negative sequences verifying for $n \geq n_0$
    \[
        r_n \leq \left( 1 - \gamma_n \right) r_{n-1} + \gamma_n a_n \left( s_n + r_n \right),
    \]
    where $\gamma_n \coloneq c_\gamma n^{-\gamma}$ with $\gamma \in (1/2, 1)$ and $c_\gamma > 0$.
    Assume that $(a_n)$ is non-increasing and converging to $0$,
    and that $s_n = \mathcal{O}(v_n)$ with $(v_n)$ positive and non-decreasing.
    Then, for any $0 < \kappa < c_\gamma(1 - 2^{\gamma - 1}) / (1 - \gamma) $, we have
    \[
        r_n = \mathcal{O} \left( \left( \exp(-\kappa n^{1-\gamma}) + a_{\lceil n/2 \rceil} \right) v_n \right).
    \]
\end{lemma}
\begin{proof}
    Let $n_1 \geq n_0$ be such that $a_{n_1} \leq 1/2$, and $M \geq 0$ such that for any $n \in \mathbb N, s_n \leq M v_n$.
    For any $n \geq n_1$:
    \[
        r_{n} \leq \left( 1 - \gamma_{n} \right) r_{n-1} + M \gamma_{n} a_{n} v_{n} + \frac{1}{2} \gamma_n r_n
        \leq \left( 1 - \frac{1}{2} \gamma_{n} \right) r_{n-1} + M \gamma_{n} a_{n} v_{n} .
    \]
    Since $(v_n)$ is positive and non-decreasing, we have for $n\geq n_1$:
    \[
        \frac{r_{n}}{v_{n}} \leq \left( 1 - \frac{1}{2} \gamma_{n} \right) \frac{r_{n-1}}{v_{n}} + M \gamma_{n} a_{n}
        \leq \left( 1 - \frac{1}{2} \gamma_{n} \right) \frac{r_{n-1}}{v_{n-1}} + \frac{1}{2} \gamma_{n} \left( 2 M a_{n} \right).
    \]
    We can apply Proposition \ref{prop::rate_deterministic_rec} to $(r_{n}/v_{n})$,
    and we obtain for any $\kappa < c_\gamma(1-2^{\gamma-1}) (1-\gamma)^{-1}$:
    \[
        \frac{r_{n}}{v_{n}}
        = \mathcal{O} \left( \exp \left( - \kappa n^{1-\gamma} \right)  + a_{\lceil n/2 \rceil} \right).
    \]
\end{proof}

\begin{proposition}\label{prop::rate_deterministic_rec}
    Let $(r_n)_{n\geq0}$, $(\gamma_n)_{n\geq1}$, and $(a_n)_{n\geq1}$ be three non-negative sequences verifying for $n \geq n_0$:
    \begin{align} \label{eq::recursive relation determinist}
        r_{n} \leq \left( 1 -  \gamma_{n}  \right) r_{n-1} + \gamma_{n} a_n.
    \end{align}
    Assume that $(a_n)$ is non increasing.
    Then, for all $n,m$ with $n_0 < m \leq n$, we have the upper bound:
    \[ r_n \leq \exp \left( - \sum_{k=m}^n \gamma_{k} \right) \left( r_{n_0} + \sum_{k=n_0+1}^{m-1} \gamma_{k}a_{k} \right) + a_{m}  . \]
    In particular, if \(\gamma_n = c / (n^{\gamma} + c')\) with \(c > 0\), \(\gamma \in (0, 1)\), \(c' \geq 0\),
    and if \(\lim_{n \to \infty} a_n = 0\),
    then for any $\kappa$ such that $0 < \kappa < c(1 - 2^{\gamma-1}) (1-\gamma)^{-1}$, we have:
    \[ r_n = \mathcal O \left( \exp(-\kappa n^{1-\gamma}) + a_{\lceil \frac{n}{2} \rceil } \right) \]
\end{proposition}

\begin{proof}
    By induction on $n$, and by splitting the sum at any rank $n_0 < m \leq n$, one has
    \[
        r_n \leq \prod_{j=n_0+1}^{n} \left( 1-  \gamma_{j} \right) r_{{n_0}}
        + \sum_{k=n_0+1}^{m-1} \left( \prod_{j=k+1}^n \left(1-  \gamma_{j} \right) \right) \gamma_{k}a_{k}
        + \sum_{k=m}^n \left( \prod_{j=k+1}^n \left(1-  \gamma_{j} \right) \right) \gamma_{k}a_{k}
    \]
    Using that $1 - x \leq e^{-x}$, that $0 \leq 1-\gamma_j \leq 1$ and that $(a_n)$ is non increasing, we get
    \[
        r_n \leq \exp \left( -\sum_{j=n_0+1}^n \gamma_j \right) r_{n_0}
        \,+\, \exp \left( - \sum_{j=m}^n \gamma_j \right) \sum_{k=n_0+1}^{m-1} \gamma_k a_k
        \,+\, \underbrace{a_m \sum_{k=m}^n \left( \prod_{j=k+1}^n (1 - \gamma_j) \right) \gamma_k }_{\eqcolon S}
    \]
    By identifying a telescopic sum, we get:
    \begin{align*}
        S & = a_{m} \sum_{k=m}^n \left( \prod_{j=k+1}^n (1 - \gamma_j) - \prod_{j=k}^n (1 - \gamma_j)  \right)
        = a_{m} \left( 1 - \prod_{j=m}^n (1 - \gamma_j) \right)
        \leq a_{m},
    \end{align*}
    which proves the first claim.

    For the particular case, assume \(\gamma_n = c / (n^{\gamma} +c') \) with \(c > 0\), \(c' \geq 0\), \(\gamma \in (0, 1)\), and \(\lim_{n \to \infty} a_n = 0\).
    Then, \(\sum_{k=m}^n \gamma_k \sim c (1 - \gamma)^{-1} (n^{1-\gamma} - m^{1-\gamma})\), and choosing \(m = \lceil n/2 \rceil\), the sum grows as \(\kappa n^{1-\gamma}\)
    with \(\kappa \coloneq c(1-\gamma)^{-1} (1-2^{\gamma-1}) \) positive.
    The exponential \(\exp \left( - \sum_{k=m}^n \gamma_k \right)\) decays as \(\exp(-\kappa n^{1-\gamma} + \mathrm{const})\).
    The bound becomes
    \[
        r_n = \mathcal{O} \left( \exp \left( -\kappa n^{1-\gamma} \right) 
        \left(1 + n^{1-\gamma} \right) + a_{\lceil n/2 \rceil} \right)
        \subset \mathcal{O} \left( \exp(-\kappa' n^{1-\gamma}) + a_{\lceil n/2 \rceil} \right)
    \]
    for $\kappa'<\kappa$, completing the proof.
\end{proof}
The following Lemma \ref{Robbins-Siegmund} is a simple corollary of the Robbins-Siegmund theorem,
and is stated and proved in \cite{godichon-baggioni_online_2025}.
\begin{lemma}\label{Robbins-Siegmund}
    Let $(V_n)$, $(B_n)$, $(E_n)$, $(D_n)$ and $(a_n)$ be five positive sequences adapted to $(\mathcal F_n)_{n \in \mathbb N}$ such that
    \[ \mathbb E\left[ V_n \mid\mathcal F_{n-1} \right] \leq (1 + E_n) V_{n-1} + B_n - D_n \quad \text{a.s.} \]
    Assume also that $\sum_{n=0}^\infty \frac{B_n}{a_n} < +\infty$ a.s. If $a_n \rightarrow \infty$, then  $V_n = o(a_n)$ a.s.
\end{lemma}

 \section{Assumptions for the streaming case}\label{ass::B}

The objective of this section is to give the assumptions for the Corollary \ref{cor::streaming SNA},
directly stated on the function $f$.
More precisely, in the case where $F(\theta) = \mathbb{E}_{\xi}  [ f(\xi , \theta) ]$ 
and considering the i.i.d. or mini-batch setting,
the assumptions on the oracles are satisfied as soon as corresponding assumptions on the function $f$ are satisfied:

 \begin{enumerate}[label=\textbf{(B\arabic*)}, series=assumptions]
 	\item \label{Assumption oracles unbiased prime} \emph{(Differentiability of $f$).}
 		For all $\theta \in \mathbb{R}^{d}$, the random function $\theta \mapsto f(\xi, \theta)$ 
		is almost surely twice differentiable at $\theta$.
 	\item \label{Assumption expected smoothness prime} \emph{(Growth Condition).}
 		There exists $  \mathcal{L}_g, \mathcal{L}_h, \sigma^2 \geq 0$ such that for all  $\theta \in \R^d$:
 		\begin{enumerate}[label=\alph*)] 
 			\item \label{Assumption:expected_smoothness_gradient prime} 
 				$\E \left[ \|\nabla f(\xi,\theta) \|^2 \right]
 				\leq  2\mathcal{L}_g \left( F(\theta) - F(\theta^*) \right) + \sigma^2$ 
 			\item \label{Assumption:expected_smoothness_hessian prime}
 				$\E \left[ \norm{\nabla^{2}  f(\xi ,\theta)}_{op}^2   \right] \leq \mathcal{L}_h$.
 		\end{enumerate}
 	\item \label{Assumption fct: hessian positive continuous prime} \emph{(Hessian at minimizer).}
 		The Hessian matrix $H \coloneq \nabla^2 F(\theta^*)$ at the minimizer is positive definite,
 		and the mapping $\theta	\mapsto \nabla^2 F(\theta)$ is continuous at $\theta^*$.
  	\item \label{Assumption oracle lyapunov prime} \emph{(Lyapunov Conditions).}
 		There exist $q , q' > 2$ and $M, M'>0$ such that:
 		\begin{enumerate}[label=\alph*)] 
 			\item \label{Assumption:Lyapunov_g prime} 
 				$ \displaystyle  \sup_{\norm{\theta - \theta^*} < M}
				\E \left[ \norm{\nabla f(\xi,\theta) - \nabla F(\theta)}^{q}    \right]
 				  < + \infty $
 			\item \label{Assumption:Lyapunov_H prime} 
 				$ \displaystyle   \sup_{\norm{\theta - \theta^*} < M'}
				\E \left[ \norm{\nabla^{2}f(\xi,\theta) - \nabla^2 F(\theta)}_{F}^{q'}  \right]
 				< + \infty $.
 		\end{enumerate}
 	\item \label{Assumption oracle variance limit prime} \emph{(Continuity of Covariance).}
 		The covariance function $\theta \mapsto \mathrm{Cov} \left( \nabla f(\xi, \theta) \right)$
 		is continuous at $\theta^*$, and we define 
 		$ \Sigma \coloneq \mathrm{Cov} \left( \nabla f(\xi, \theta^*) \right)$,
		which, as $\nabla F(\theta^*)=0$, is equivalent to 
		$\Sigma = \mathbb{E} \left[ \nabla f(\xi, \theta^*) \nabla f(\xi, \theta^*)^T \right]$.
 \end{enumerate}
{ 
\section{Matrix gradient}\label{sec:gradient}

\subsection*{Gradient and adjoint in a Hilbert space}

Let $(\mathcal{H}, \langle \cdot , \cdot \rangle)$ be a Hilbert space, 
$\| \cdot \|$ the norm induced by the inner product, $U \subset \mathcal{H}$ an open set,
and $f: U \to \mathbb{R}$ a function differentiable at a point $a \in U$.
The differential $df(a)$ being by definition a continuous linear operator from $\mathcal{H}$ to $\mathbb{R}$,
the Riesz representation theorem ensures the existence of a unique element $\nabla f(a) \in \mathcal{H}$
such that
\[
    \forall h \in \mathcal{H}, \quad df(a)(h) = \langle \nabla f(a) , h \rangle.
\]
The element $\nabla f(a)$ is called the gradient of $f$ at $a$.

Let $L: \mathcal{H} \to \mathcal{H}$ be a continuous linear operator.
For any $y \in \mathcal{H}$, the mapping $h \mapsto \langle y , L h \rangle$ 
is a bounded linear functional on $\mathcal{H}$,
so by the Riesz representation theorem, there exists a unique element $L^{*} y \in \mathcal{H}$
such that
\[
    \forall h \in \mathcal{H}, \quad \langle y , L h \rangle = \langle L^{*} y , h \rangle. 
\]
The operator $L^{*}: \mathcal{H} \to \mathcal{H}$ is called the adjoint operator of $L$.

\subsection*{Gradient of a quadratic functional}

Let $L: \mathcal{H} \to \mathcal{H}$ be a continuous linear operator, and $b \in \mathcal{H}$.
Consider the quadratic functional $f: \mathcal{H} \to \mathbb{R}$ defined as
\[
    f(x) \coloneq \| L x + b \|^{2}.
\]
Let $x, h \in \mathcal{H}$. We have
\begin{align*}
    f(x + h)
    &= \| L (x + h) + b \|^{2} \\
    &= \| L x + b \|^{2} + 2 \langle L x + b , L h \rangle + \| L h \|^{2} \\
    &= f(x) + 2 \langle L^{*} (L x + b) , h \rangle + \| L h \|^{2} .
\end{align*}
The gradient of $f$ at $x$ is therefore given by
\[
    \nabla f(x) = 2 L^{*} (L x + b).
\]

\subsection*{Application to matrix spaces}

We now apply the previous result to the Hilbert space $\mathcal{H} = \mathcal{M}_{d}(\mathbb{R})$,
the space of real $d \times d$ matrices, endowed with the Frobenius inner product:
\[
    \langle M , N \rangle_F \coloneq \mathrm{tr}(M^T N).
\]
The space $\mathcal{H}$ is finite-dimensional, so all linear operators on $\mathcal{H}$ are continuous.
Let $A,B \in \mathcal{H}$.
The adjoint of the linear operator $M \mapsto A M$ is $N \mapsto A^T N$.
Indeed, for any $M, N \in \mathcal{H}$,
\[
    \langle N , A M \rangle_F 
    = \mathrm{tr}(N^T A M) = \mathrm{tr}((A^T N)^T M) = \langle A^T N , M \rangle_F.
\]
Consider the quadratic functional $f: \mathcal{H} \to \mathbb{R}$ defined as
\[
    f(M) \coloneq \| A M + B \|_F^{2}.
\]
The gradient of $f$ at point $M \in \mathcal{H}$ is given by
\[
    \nabla f(M) = 2 A^T (A M + B).
\]

\subsubsection*{Gradient of composition with transpose}

Let $f: \mathcal{H} \to \mathbb{R}$ be a differentiable function, and define $g: \mathcal{H} \to \mathbb{R}$ as
\[
    g(M) \coloneq f(M^T).
\]
Denoting by $t: \mathcal{H} \to \mathcal{H}$ the transpose operator for matrices, we have $g = f \circ t$.
The transpose operator is linear and self-adjoint since for any $M, N \in \mathcal{H}$,
\[
    \langle N , t(M) \rangle_F 
    = \mathrm{tr}(N^T M^T) 
    = \mathrm{tr}((N^T)^T M) 
    = \langle t(N) , M \rangle_F.
\]
Using the chain rule for differentiation, we obtain that the differential of $g$ at point $M \in \mathcal{H}$ is given by
\[
    dg(M) = df(M^T) \circ t.
\]
Therefore, the gradient of $g$ at point $M \in \mathcal{H}$ is
\[
    \nabla g(M) 
    = t^*(\nabla f(M^T)) 
    = t(\nabla f(M^T)) 
    = (\nabla f(M^T))^T. 
\]    
}

\end{document}